\documentclass[a4paper,11pt]{article}
\usepackage[latin1]{inputenc}
\usepackage[T1]{fontenc}
\usepackage[british,UKenglish,USenglish,american]{babel}
\usepackage{amsmath}
\usepackage{amsfonts}
\usepackage{hyperref}
\usepackage[T1]{fontenc}
\usepackage[latin1]{inputenc}
\usepackage{theorem}
\usepackage{graphics}
\usepackage{makeidx}
\usepackage{fancyhdr}
\usepackage{bbm}
\usepackage{bm} 
\usepackage{amssymb}
\usepackage{color}
\numberwithin{equation}{section}
\usepackage{amssymb}
\newcommand{\be}{\begin{equation}}
\newcommand{\ee}{\end{equation}}
\newcommand{\benn}{\begin{equation*}}
\newcommand{\eenn}{\end{equation*}}
\newcommand{\bea}{\begin{eqnarray}}
\newcommand{\eea}{\end{eqnarray}}
\newcommand{\beann}{\begin{eqnarray*}}
\newcommand{\eeann}{\end{eqnarray*}}

\newtheorem{theorem}{Theorem}[section]

\newtheorem{corollary}[theorem]{Corollary}
\newtheorem{lemma}[theorem]{Lemma}

\newtheorem{definition}[theorem]{Definition}
\theorembodyfont{\rm}

\newtheorem{remark}[theorem]{Remark}
\newtheorem{example}[theorem]{Example}
\newtheorem{assumptions}[theorem]{Assumptions}

\newcommand{\qed}{\hfill $\Box$\smallskip}

\def\R{\mathbb{R}}

\def\N{\mathbb{N}}

\def\P{\mathbb{P}}

\newcommand{\B}{\mathbb{B}} 
\newcommand{\W}{\mathbf{W}}

\newcommand{\WW}{\mathbb{W}}

\newcommand{\gubnorm}[2]{\left\|#1,#2 \right\|_{W,2\gamma,\alpha}}
\newcommand{\gubnormpar}[3]{\left\|#1,#2 \right\|_{W,2\gamma,#3}}

\def\cB{\mathcal{B}}
\def\cC{\mathcal{C}}
\def\cD{\mathcal{D}}

\def\cF{\mathcal{F}}

\def\cH{\mathcal{H}}

\def\cL{\mathcal{L}}
\def\cM{\mathcal{M}}

\def\cO{\mathcal{O}}

\def\cR{\mathcal{R}}

\def\cX{\mathcal{X}}

\def\txtc{{\textnormal{c}}}
\def\txtd{{\textnormal{d}}}
\def\txte{{\textnormal{e}}}

\def\txts{{\textnormal{s}}}

\def\txtD{{\textnormal{D}}}

\def\Id{{\textnormal{Id}}}

\def\ra{\rightarrow}
\def\I{\infty}

\title{Center Manifolds for Rough Partial Differential Equations}
\author{Christian Kuehn\thanks{Technical University of Munich, 
		Department of Mathematics, 85748 Garching bei M\"unchen, Germany. E-Mail: ckuehn@ma.tum.de}~~and~Alexandra 
	Neam\c tu\thanks{University of Konstanz, Department of Mathematics and Statistics, 
		Universit\"atsstra\ss{}e 10 78464 Konstanz, Germany. E-Mail: alexandra.neamtu@uni-konstanz.de}}

\begin{document}

\maketitle

\begin{abstract}
We prove a center manifold theorem for rough partial differential equations (rough PDEs). The class of rough PDEs we consider contains as a key subclass reaction-diffusion equations driven by nonlinear multiplicative noise, where the stochastic forcing is given by a $\gamma$-H\"older rough path, for $\gamma\in(1/3,1/2]$. Our proof technique relies upon the theory of rough paths and analytic semigroups in combination with a discretized Lyapunov-Perron-type method in a suitable scale of interpolation spaces. The resulting center manifold is a random manifold in the sense of the theory of random dynamical systems (RDS). We also illustrate our main theorem for reaction-diffusion equations as well as for the Swift-Hohenberg equation.
\end{abstract}

\medskip


\section{Introduction}
\label{sec:intro}

Center manifolds, as well as the easier stable and unstable manifolds, are key technical tools in dynamical systems theory~\cite{GH}. The idea is to split the dynamics into exponentially attracting, exponentially repelling and neutral directions near a steady state. This splitting can often be obtained locally on the level of a linearized system. If the linearized operator has no spectrum on the imaginary axis then the steady state is called hyperbolic. In the hyperbolic situation, quite classical stable and unstable manifold theory as well as local topological equivalence between the linearized and the nonlinear system exist for many classes of evolution equations~\cite{GH,Henry,BatesJones}. In the non-hyperbolic situation, when spectrum on the imaginary axis appears, we actually need more involved center manifold theory~\cite{Carr}. Although this situation may appear non-generic at first, it is well-understood that it is generic in differential equations with parameters, where it is of crucial importance to obtain center manifolds to study bifurcation problems~\cite{GH,Kuznetsov}. Furthermore, center manifold theory can yield effective dimension reduction near a steady state if there are only attracting and center directions, which is a concept that can be extended to entire manifolds of steady states, e.g., in the context of slow manifolds for multiple time scale systems~\cite{Fenichel,Kuehn}.\medskip

For stochastic differential equations, there already some results regarding center manifolds, see~\cite{Arnold,Boxler1,Boxler2,Roberts} for stochastic ordinary differential equations (SODEs) and~\cite{ChenRobertsDuan,ChekrounLiuWang, DuanWang,BloemkerWang} for SPDEs. However, stochastic center manifold theory is still far less well-developed in comparison to deterministic ordinary differential equations (ODEs) or partial differential equations (PDEs). The aim of this work is to investigate center manifolds for semilinear rough evolution equations, where our main motivation arises from semilinear reaction-diffusion SPDEs, which are included as a particular subclass in our results. Our work here extends our earlier results obtained in the finite-dimensional setting established in~\cite{KN}. Moreover we emphasize that these results yield the existence of center manifolds in suitable interpolation spaces, which naturally arise in the context of parabolic PDEs. Therefore, this abstract framework is not restricted to Hilbert spaces. In summary, this work allows us to substantially extend the theory developed in~\cite{KN} as well as the invariant manifold theory for SPDEs~\cite{DuanLuSchmalfuss, GarridoLuSchmalfuss,CaraballoDuanLuSchmalfuss,ChenRobertsDuan, Neamtu}. \medskip 

There are several major technical difficulties one encounters, when trying to establish center manifold results for stochastic partial differential equations (SPDEs). A first conceptual difficulty is to employ the concept of random dynamical systems (RDS)~\cite{Arnold} for SPDEs. It is well-known that an It\^{o}-type SODE generates an RDS under reasonable assumptions~\cite{Arnold,Kunita,Scheutzow}. However, the generation of an RDS from an It\^{o}-type SPDE has been a long-standing open problem, mostly since Kolmogorov's theorem breaks down for random fields parametrized by infinite-dimensional Hilbert spaces~\cite{Mohammed}. As a consequence it is not trivial, how to obtain a RDS from a general SPDE. This problem was fully solved only under very restrictive assumptions on the structure of the noise driving the SPDE.~For instance, if one deals with purely additive noise or certain particular multiplicative Stratonovich noise, there are standard transformations which reduce the SPDE to a random PDE.~Since this random PDE can be solved pathwise it is straightforward to obtain an RDS. However, for nonlinear multiplicative noise, this technique is no longer applicable, not even if the random input is a Brownian motion.~As a consequence of this issue, dynamical aspects for SPDEs, e.g.~invariant manifolds, have not been investigated in their full generality. In the finite-dimensional case there are results concerning invariant manifolds for delay equations using rough paths~\cite{Riedel1, Riedel2} and center manifolds in~\cite{KN}. Rough path techniques provide a very natural framework to obtain RDS from SPDEs driven by  general multiplicative noise since the usual problems with the nullsets do not appear in a pathwise approach. For instance, there are results regarding the existence of random dynamical systems generated by rough PDEs
with transport~\cite{Hofmanova1,Hofmanova2}, nonlinear multiplicative~\cite{HesseNeamtu2} and nonlinear conservative noise~\cite{FehrmanGess}.  In this work we go beyond the of existence of RDS in the infinite-dimensional setting and establish a center manifold theorem (Theorem~\ref{manifold}).  \\

The second main, more technical, obstacle we encounter is due to the fact that one wants to include the case, when the analytic semigroup $(S_t)_{t\geq 0}$ generated by the linear part of the SPDE is no longer H\"older continuous in zero. However, this regularity is required in order to introduce the rough convolution and to obtain an expansion of the solution in terms of H\"older-continuous functions. There are several approaches to deal with this problem~\cite{Gubinelli,GubinelliLejayTindel,GubinelliTindel}, or to work with modified H\"older spaces which compensate the time-singularity in zero, as in~\cite{HesseNeamtu1}, or to consider more space-regularity in order to compensate the missing time-regularity as in~\cite{GHairer}. 
More precisely,  in order to define the rough convolution
$\int_0^t S(t-s)Y_s~\txtd \W_s$ with respect to a $\gamma$-H\"older rough path $\W=(W,\mathbb{W})$, one needs the notion of a controlled rough path~\cite{Gubinelli}, which is a pair $(Y,Y')$ of $\gamma$-H\"older continuous functions
satisfying an abstract Taylor-like expansion in terms of H\"older regularity
given by
\begin{align*}
Y_t= Y_s+Y'_s W_{s,t}+ R^{Y}_{s,t},
\end{align*}
where the remainder $R^Y_{s,t}$ is supposed to be $2\gamma$-H\"older-regular. Due to the lack of regularity of the semigroup $(S_t)_{t\geq 0}$ in zero, it is a challenging task to find an appropriate meaning of a controlled rough path. The main idea is to consider controlled rough paths on a scale of Banach spaces $(\cB_\alpha)_{\alpha\in\mathbb{R}}$ satisfying the following interpolation inequality which means that
\begin{align*}
|x|^{\alpha_3-\alpha_1}_{\alpha_2} \lesssim |x|^{\alpha_3-\alpha_2}_{\alpha_1} |x|^{\alpha_2-\alpha_1}_{\alpha_3}
\end{align*}
 for $\alpha_1\leq \alpha_2\leq \alpha_3$ and $x\in  \cB_{\alpha_3}$.
The advantage of this approach is that it allows one to view the semigroup as a bounded operator on all these spaces and exploit space-time regularity specific to the parabolic setting. Such an approach was exploited in~\cite{GHN} in the context of non-autonomous rough PDEs and in~\cite{GHairer},
where the semigroup was directly incorporated in the definition of the controlled rough path. Following the approach of controlled rough paths in interpolation spaces, we manage to prove the existence of center manifolds for parabolic rough PDEs based on the Lyapunov-Perron method. This is the key analytical contribution of this work.\\

The paper is structured as follows: In Section~\ref{heuristics}, we provide an overview regarding our setting, compare it formally to an existing approach, and we motivate how to set up the iteration procedure to construct the center manifold. In Section~\ref{preliminaries} we collect preliminaries concerning evolution equations and controlled rough paths. In Section~\ref{sect:est:sol}, we prove a-priori estimates for the solution of rough evolution equations, which will be key components to justify the existence of a fixed point for the Lyapunov-Perron method. In Section~\ref{sect:rds}, we present the background from RDS, construct suitable random cocycles, and define random center invariant manifolds. In Section~\ref{lcm}, we finally set up a discrete Lyapunov-Perron method and prove the existence of a center manifold. We present some examples in Section~\ref{appl}.\medskip

\textbf{Acknowledgments:} CK acknowledges support by a Lichtenberg 
Professorship.   AN thanks Felix Hummel for helpful discussions regarding interpolation spaces.

\section{Heuristic Overview}
\label{heuristics}

This section provides an overview of the existing foundations for center manifold theory for stochastic partial differential equations and it describes the main goal and the strategy of this work.

\subsection{A Classical Construction}
\label{sec:classical}

In the context of random invariant sets, certain classes of SPDEs have been studied in the literature; see~\cite{Arnold,DuanLuSchmalfuss,ChekrounLiuWang,ChenRobertsDuan} and the references therein. We recall the main ideas that have been used so far in proving the existence of random center manifolds for such SPDEs, given by
\begin{equation}
\label{eqstrat}
\begin{cases}
\txtd u = (A u + f (u)) ~\txtd t + u \circ \txtd \tilde{B}_{t}\\
u(0)=\xi
\end{cases}
\end{equation}
on a separable Hilbert space $H$. Here the linear operator $A$ generates a $C_{0}$-semigroup $(S_t)_{t\geq 0}$ on $H$, $f$ is a locally Lipschitz nonlinear term with $f(0)=0=f'(0)$, and $\tilde{B}$ denotes a two-sided real-valued Brownian motion; note that $f(0)=0$ ensures that $u=0$ is a steady state of~\eqref{eqstrat}. Suppose the spectrum of the linear operator $A$ consists of finitely many eigenvalues with zero real part, and all other eigenvalues have strictly negative real parts, i.e.~$\sigma(A)=\sigma^{\txtc}(A)\cup \sigma^{\txts}(A)$, where $\sigma^{\txtc}(A)=\{\lambda \in \sigma(A)\mbox{ : } \mbox{Re}(\lambda)=0\}$ and $\sigma^{\txts}(A)=\{\lambda\in \sigma(A) \mbox{ : } \mbox{Re}(\lambda)<0 \}$. The subspaces generated by the eigenvectors corresponding to these eigenvalues are denoted by $H^{\txtc}$ respectively $H^{\txts}$ and are referred to as {\em center} and {\em stable} subspace. These subspaces provide an invariant splitting of $H=H^{\txtc}\oplus H^{\txts}$.
We denote the restrictions of $A$ on $H^{\txtc}$ and $H^{\txts}$ by $A_{\txtc}:=A|_{\cH^{\txtc}}$ and $A_{\txts}:=A|_{\cH^{\txts}}$. Since $H^{\txtc}$ is finite-dimensional we obtain that $S^{\txtc}(t):=\txte^{tA_{\txtc}}$ is a {\em group} of linear operators on $H^{\txtc}$. Moreover, there exist projections $P^{\txtc}$ and $P^{\txts}$ such that $P^{\txtc} + P^{\txts} = \mbox{Id}_{H}$ and $A_{\txtc} =A|_{\cR(P^{\txtc})}$ and $A_{\txts}=A|_{\cR(P^{\txts})}$, where $\cR$ denotes the range of the corresponding projection. Additionally, we impose the following dichotomy condition on the semigroup. We assume that there exist two exponents $\gamma$ and $\beta$ with $-\beta^*<0\leq \gamma^*<\beta^*$ and constants $M_{\txtc},M_{\txts}\geq 1$, such that
\begin{align}
&\|S^{\txtc}(t)  x\|_{H} \leq M_{\txtc} \txte^{{\gamma^*} t} \|x\|_{H}, 
~~~\mbox{  for } t\leq 0 \mbox{ and } x\in H;\label{gamma:h}\\
& \|S^{\txts}(t) x\|_{H} \leq M_{\txts} \txte^{-{\beta^*} t} \|x\|_{H}, \label{beta:h}
~~\mbox{for } t\geq 0 \mbox{ and } x\in H.
\end{align}
Furthermore, we introduce the stationary Ornstein-Uhlenbeck process, i.e.~the stationary solution of the Langevin equation
$$ \txtd z_t =-z~\txtd t + \txtd \tilde{B}_t,$$
which is given by
$$ z(\theta_{t}\tilde{B})=\int\limits_{-\infty}^{t} \txte^{-(t-s)}~\txtd\tilde{B}_s=\int\limits_{-\infty}^{0} \txte^{s}~\txtd\theta_{t}\tilde{B}_{s}.$$
Here $\theta$ denotes the usual Wiener-shift, i.e.~$\theta_{t}{\tilde{B}}_{s}:=\tilde{B}_{t+s} -\tilde{B}_{t}$ for $s,t\in\mathbb{R}$. In this case, using the Doss-Sussmann transformation $u^{*}:=u\txte^{-z(\tilde{B})}$, the SPDE~\eqref{eqstrat} reduces to the non-autonomous random differential equation
\begin{equation}\label{ou}
\txtd u =( A u + z(\theta_{t}{\tilde{B}}) u + g (\theta_{t}{\tilde{B}},u) )~\txtd t,
\end{equation}
where we dropped the $*$-notation and set $g(\tilde{B},u):=\txte^{-z(\tilde{B})}f(\txte^{z(\tilde{B})}u)$. Note that no stochastic integrals appear in~\eqref{ou} and one can prove the existence of center manifolds for~\eqref{ou} almost like in the deterministic setting, using the Lyapunov-Perron method. More precisely, one infers that the continuous-time Lyapunov-Perron transform for~\eqref{ou} is given by
\begin{align}
\label{lpeinfach}
J({\tilde{B}},u,\xi)[t] & := S^{\txtc}_t \txte^{\int\limits_{0}^{t} 
	z(\theta_{\tau}{\tilde{B}}) ~\txtd \tau} P^{\txtc}\xi   
+ \int\limits_{0}^{t} S^{\txtc}_{t-r}  \txte^{\int\limits_{r}^{t} 
	z(\theta_{\tau}{\tilde{B}}) ~\txtd \tau} P^{\txtc} 
g(\theta_{r}{\tilde{B}}, u(r))~\txtd r\nonumber\\
& +\int\limits_{-\infty}^{t} S^{\txts}_{t-r} \txte^{\int\limits_{r}^{t} 
	z(\theta_{\tau}{\tilde{B}}) ~\txtd \tau} P^{\txts} g(\theta_{r}{\tilde{B}}, 
u(r))~\txtd r.
\end{align}
Further details regarding this operator can be found in~\cite{WanngDuan},~\cite[Sec.~6.2.2]{DuanWang},~\cite[Ch.4]{ChekrounLiuWang} and the references specified therein. The next natural step is to show that~\eqref{lpeinfach} possesses a fixed-point in a certain function space. One possible choice turns out to be $BC^{\eta,z}(\mathbb{R}^{-};H)$, see~\cite[p.~156]{DuanWang}. This space is defined as
\benn
BC^{\eta,z}(\mathbb{R}^{-};H):=\left\{u:\mathbb{R}^{-}\to H, 
~u ~\mbox{is continuous and } \sup\limits_{t\leq 0}\txte^{-\eta t 
	-\int\limits_{0}^{t}z(\theta_{\tau}{\tilde{B}})~\txtd \tau } 
\|u(t)\|_{H}<\infty\right\}
\eenn
and is endowed with the norm
\begin{align*}
||u||_{BC^{\eta,z}} := \sup\limits_{t\leq 0}~\txte^{-\eta t 
	-\int\limits_{0}^{t}z(\theta_{\tau}{\tilde{B}})~\txtd \tau }\|u(t)\|_{H}.
\end{align*}
Here $\eta$ is determined from~\eqref{gamma:h} and~\eqref{beta:h}, namely one has $-\beta^*<\eta<0$. Note that the previous expressions are well-defined since 
\benn
\lim\limits_{t\to \pm\infty}\frac{|z(\theta_t{\tilde{B}})|}{|t|}=0, 
\eenn
according to~\cite[Lem.~2.1]{DuanLuSchmalfuss} and the references specified therein. Under a suitable smallness assumption on the Lipschitz constant of $f$ (gap condition) one can show that $J$ possesses a fixed-point $\Gamma(\cdot,\tilde{B},\xi)$ for $\xi\in H^{\txtc}$. Since a global Lipschitz condition on $f$ is quite restrictive in applications, one usually introduces a cut-off function to truncate the nonlinearity outside a random ball around the origin. This fixed-point characterizes the random center manifold $M^{\txtc}(\tilde{B})$ for~\eqref{ou}. More precisely, one can show that $M^{\txtc} (\tilde{B})$ can be represented by the graph of a function $h^{\txtc} (\tilde{B},\cdot)$, where $h^{\txtc} (\tilde{B},\xi)=P^{\txts}\Gamma(0,\tilde{B},\xi)$, i.e.
\begin{align}\label{h:strat}
h^\txtc (\tilde{B},\xi) = \int\limits_{-\infty}^{0} S^{\txts}_{-\tau} \txte^{\int\limits_{\tau}^{0}z(\theta_{r}\tilde{B})~\txtd r} P^{\txts} g (\theta_{\tau}\tilde{B},\Gamma(\tau,\tilde{B},\xi))~\txtd \tau,~~\mbox{ for  } \xi \in H^{\txtc} \cap B(0,\rho(\tilde{B})).
\end{align}
Here $B(0,\rho(\tilde{B}))$ denotes a random neighborhood of the origin, i.e.~the radius $\rho(\tilde{B})$ depends on the intensity/magnitude of the noise.

\begin{example}
\label{ex:cm:strat}
We now illustrate via a computational example, how the theory briefly introduced above can be applied. Let $a$ and $\sigma$ stand for two positive parameters and consider the reaction-difussion SPDE on $H:=L^{2}(0,\pi)$
\begin{align}\label{ex1}
\begin{cases}
	\txtd u = (\Delta u + u - a u^{3}) ~ dt + \sigma u \circ \txtd \tilde{B}_{t} \\
	u(0,t)=u(\pi,t)=0, ~\mbox{for } t \geq 0\\
	u(x,0)=u_{0}(x),~~~~~~~\mbox{for  } x\in(0,\pi).
\end{cases}
\end{align}
 Substituting $u:=u^{*}\txte^{\sigma z(\tilde{B})}$ and dropping the $*$-notation, we obtain the non-autonomous PDE with random coefficients
\begin{align}\label{random:pde}
\frac{\partial u}{\partial t} = \Delta u + u + \sigma z(\theta_{t}\tilde{B}) u - a \txte^{2\sigma z (\theta_{t}\tilde{B})}u^{3}.
\end{align}
The spectrum of $$A u:=\Delta u + u$$ with domain $D(A)=H^2(0,\pi)\cap H^{1}_{0}(0,\pi)$ is given by $\{1-n^{2}\mbox{ : } n\geq 1\}$ with corresponding eigenvectors $\{\sin(nx)\mbox{ : } n\geq 1\}$. The eigenvectors give us the center subspace $H^{\txtc} =\mbox{span}\{\sin x\}$ and the stable one  $H^{s}=\mbox{span}\{\sin(nx)\mbox{ : } n\geq 2\}$. Based upon the previous considerations, one can infer that~\eqref{ex1} has a local center manifold
$$M^{\txtc} (\tilde{B}) =\{b\sin x + h^{\txtc}(\tilde{B}, b \sin x)\}=\Bigg\{ b\sin x +\sum\limits_{n=2}^{\infty} c_{n}(\tilde{B},b)\sin (nx) \Bigg\}.$$
In this case, it is also possible to derive suitable approximation results for $h^{\txtc}$, namely one can show that $c_{n}(\tilde{B},b)=\mathcal{O}(b^{3})$ as $b\to 0$. Plugging this in~\eqref{random:pde} gives us a {\em non-autonomous} random ODE on the {\em center manifold}
\begin{align*}
\frac{\txtd b}{\txtd t} = \sigma z(\theta_{t}\tilde{B}) b - \frac{3}{4} a b^3 \txte^{2\sigma z (\theta_{t}\tilde{B})}+  \cO(b^{5}).
\end{align*}
Since $-u$ is also a solution for~\eqref{random:pde} we have that $c_{n}(\tilde{B},b)=0$ for $n$ even. Therefore, one has the following approximation of $h$
\begin{align*}
	h^{\txtc}(\tilde{B}, b\sin x) = c_{3}(\tilde{B},b)\sin 3x + \cO (b^{5}).
\end{align*}
\end{example}

\subsection{Our Goal}

The approach presented in the previous section is somewhat limited in applicability due to use of a Doss-Sussmann transformation to a non-autonomous random PDE. The main goal of this work is to extend center manifold theory for SPDEs to equations driven by rough noise and to recover the Stratonovich case mentioned above as a special case. More precisely, we aim to obtain center manifolds for equations of the form
\begin{align}\label{h:aim}
\begin{cases}
\txtd u = (A u + F(u))~\txtd t + G(u)~\txtd W,\\
u(0)=\xi,
\end{cases}
\end{align}
where $G$ is {\em nonlinear} and the noisy input $W$ is supposed to be more \emph{irregular} than a Brownian motion, i.e.~$W\in C^{\gamma}$ for $\gamma\in(1/3,1/2)$. This includes Brownian motion but applies to a much wider class of Gaussian processes.~Examples for $G$ are polynomials with smooth coefficients (see Section~\ref{appl}), or integral operators with a smooth kernel as discussed in~\cite[Section~7]{HesseNeamtu1}. As a first step we have to rigorously prove the existence of center manifolds for~\eqref{h:aim}. Further works will be devoted to approximation results and to related problems in bifurcation theory, see for example~\cite{B1,B2}.\\

In contrast to the previous technique in Section~\ref{sec:classical}, we are \emph{not going to transform}~\eqref{h:aim} to a random PDE, but we are going to \emph{work directly} with its mild solution
\begin{align}\label{h:mild}
u_{t} =S_t \xi + \int\limits_{0}^{t} S_{t-s}F(u_{s})~\txtd s + \int\limits_{0}^{t} S_{t-s} G(u_{s})~\txtd W_{s}.
\end{align}
This is possible due to the {\em pathwise construction} of the stochastic integral in~\eqref{h:mild}. The Lyapunov-Perron map in this case is given by 
\begin{align*}
J(W,u,\xi)[t]: &= S^{\txtc}_t P^{\txtc} \xi + \int\limits_{0}^{t} S^{\txtc}_{t-r} P^{\txtc} F(u_{r})~\txtd r + \int\limits_{0}^{t} S^{\txtc}_{t-r} P^{\txtc} G(u_r)~\txtd W_{r}\\& +\int\limits_{-\infty}^{t} S^{\txts}_{t-r} P^{\txts} F (u_{r})~\txtd r + \int\limits_{-\infty}^{t}S^{\txts}_{t-r}P^{\txts} G(u_{r})~\txtd W_{r}.
\end{align*}
Due to the stochastic integrals appearing above, the technical challenge consists in finding an appropriate framework to formulate the fixed-point problem for $J$. After this is established, one should intuitively be able to show that the fixed-point $\Gamma$ of $J$ characterizes the local center manifold. More precisely, the random manifold should have a graph structure, where the function $h^{\txtc}=P^{\txts}\Gamma(0,W,\xi)$ is given by
\begin{align*}
h^{\txtc}(W,\xi) = \int\limits_{-\infty}^{0} S^{s}_{-\tau} P^{\txts} F(\Gamma(\tau, W,\xi))~\txtd \tau + \int\limits_{-\infty}^{0}S^{\txts}_{-\tau} P^{\txts} G (\Gamma(\tau,W,\xi))~\txtd W_{\tau},
\end{align*}
for $\xi \in H^{\txtc}\cap B(0,\rho(W))$ similarly to~\eqref{h:strat}. Naturally, the size of the random ball should depend on the growth of the nonlinear terms $F$ and $G$ and on the random input $W$. The following sections provide the necessary tools and rigorously prove these heuristic considerations. Moreover, the center manifold theory developed in this work is applicable to rough PDEs in interpolation spaces and it is not restricted to the Hilbert space-valued setting.


\section{Preliminaries}
\label{preliminaries}

We fix $T>0$, let $\alpha\in\R$ and consider on a scale of Banach spaces $\cB_\alpha$ the equation 
\begin{equation}~\label{rpde}
\begin{cases}
\txtd Y_t= A Y_t ~\txtd t + F(Y_t)~\txtd t + G(Y_t)~\txtd \textbf{W}_t,~~t\in[0,T]\\
Y_{0}=\xi\in \cB_\alpha,
\end{cases}
\end{equation}
where $Y:[0,T]\ra \cB_\alpha$ is the unknown, the linear part defined via the operator $A$ with domain $\cB_1:=D(A)$ generates an analytic $C_{0}$-semigroup $(S_t)_{t\geq 0}$ on a separable Banach space $\cB$. 
 The scale of Banach spaces $(\cB_\alpha)_{\alpha\in\R}$, the nonlinear drift and diffusion coefficients $F$ and $G$ will be discussed further below, and the noise $\textbf{W}$ is a $\gamma$-H\"older $d$-dimensional rough path for $\gamma\in(\frac{1}{3},\frac{1}{2})$ and some fixed $d\geq 1$. This case includes the Brownian motion and the fractional Brownian motion for $H\in(\frac{1}{3},\frac{1}{2}]$.\medskip

\textbf{Notation:} As commonly met in the rough path theory, we use the notation $Y_t$ instead of $Y(t)$ and $Y_{s,t}:=Y_{t}-Y_{s}$ stands for an increment. \\
Keeping this in mind we specify that the $d$-dimensional noisy input $W=(W^{1}, \ldots W^{d})$ is assumed to be a $\gamma$-H\"older rough path $\textbf{W}:=(W,\mathbb{W})$, for $\gamma\in(1/3,1/2]$. More precisely,
\begin{align*}
W\in C^{\gamma}([0,T];\mathbb{R}^{d}) ~~\mbox{ and } ~~ \mathbb{W}\in C^{2\gamma}([0,T]^2;\mathbb{R}^{d}\otimes\mathbb{R}^{d})
\end{align*}
and the connection between $W$ and $\mathbb{W}$ is given by Chen's relation
\begin{align}\label{chen}
\mathbb{W}_{s,t}- \mathbb{W}_{s,u}-\mathbb{W}_{u,t}=W_{s,u}\otimes W_{u,t},
\end{align}
where we used the increment notation $W_{s,t}=W_{t}-W_{s}$ introduced above. The term $\mathbb{W}$ is referred to as second-order process or L\`evy-area. The process can be interpreted as the iterated integral~\cite[Chapter 10]{FritzHairer}
\begin{align*}
\mathbb{W}_{s,t}=\int\limits_{s}^{t} (W_{r}- W_s) \otimes \txtd W_{r} .
\end{align*}
Throughout this manuscript we assume without loss of generality $d=1$ since the generalization to higher dimensions can be done component-wise and does not require any additional arguments. We further introduce an appropriate distance between two $\gamma$-H\"older rough paths.
\begin{definition}
	Let $J\subset\mathbb{R}$ be a compact interval, $\Delta_J:=\{ (s,t)\in J \times J : s\leq t \}$ and let $\mathbf{W}=(W,\WW)$ and $\mathbf{\tilde{W}}=(\tilde{W},\tilde{\WW})$ be two $\gamma$-H\"older rough paths. We introduce the $\gamma$-H\"older rough path (inhomogeneous) metric
	\begin{align}\label{rp:metric}
	d_{\gamma,J}(\mathbf{W},\mathbf{\tilde{W}} )
	:= \sup\limits_{(s,t)\in \Delta_J} \frac{|W_{s,t}-\tilde{W}_{s,t}|}{|t-s|^{\gamma}}
	+ \sup\limits_{(s,t) \in \Delta_{J}}
	\frac{|\mathbb{W}_{s,t}-\tilde{\mathbb{W}}_{s,t}|} {|t-s|^{2\gamma}}.
	\end{align}
We set $\rho_\gamma(\mathbf{W}):=d_{\gamma,[0,T]}(\mathbf{W},0)$ and denote the space of $\gamma$-H\"older rough paths by $\cC^{\gamma}([0,T];\mathbb{R})$.
\end{definition}
For more details on this topic consult~\cite[Chapter 2]{FritzHairer}. We stress that in our situation we always have that $W_0=0$ and therefore~\eqref{rp:metric} is a metric. \\

Since we consider parabolic rough PDEs, we work with the following function spaces similar to~\cite{GHairer,GHN}. These reflect a suitable interplay between space and time regularity available in the parabolic setting.
\begin{definition}\label{raum}
	A family of separable Banach spaces $(\cB_\alpha,|\cdot|_\alpha)_{\alpha\in\mathbb{R}}$ is called a monotone family of interpolation spaces if for $\alpha_1\leq \alpha_2$, the space $\cB_{\alpha_2}\subset \cB_{\alpha_1}$ with dense and continuous embedding and the following interpolation inequality holds for $\alpha_1\leq \alpha_2\leq \alpha_3$ and $x\in  \cB_{\alpha_3}$:
	\begin{align}\label{interpolation:ineq}
	|x|^{\alpha_3-\alpha_1}_{\alpha_2} \lesssim |x|^{\alpha_3-\alpha_2}_{\alpha_1} |x|^{\alpha_2-\alpha_1}_{\alpha_3}.
	\end{align}
\end{definition}
The main advantage of this approach is that we can view the semigroup $(S(t))_{t\geq 0}$ as a linear mapping between these interpolation spaces and obtain the following standard bounds for the corresponding operator norms. If $S:[0,T]\to \cL(\cB_{\alpha},\cB_{ \alpha+1})$ is such that for every $x\in \cB_{\alpha+1}$ and $t\in(0,T]$ we have that $|(S_t-\Id)x|_{\alpha} \lesssim t|x|_{\alpha+1}$ and $|S(t)x|_{\alpha+1}\lesssim t^{-1} |x|_{\alpha}$, then for every $\sigma\in[0,1]$ we have that $S(t)\in\cL(\cB_{\alpha+\sigma})$ and 
\begin{align}
|(S_t-\Id) x|_{\alpha}&\lesssim t^\sigma |x|_{\alpha+\sigma}\label{hg:1}\\
|S(t)x|_{\alpha+\sigma}&\lesssim t^{-\sigma}|x|_\alpha\label{hg:2}.
\end{align}
An example of such spaces is constituted by $\cB_\alpha=[\cB,\cB_1]_\alpha$, where $[\cdot,\cdot]_\alpha$ denotes the complex interpolation.
For further details regarding these interpolation spaces, see~\cite{Lunardi,GHN}. \\

 
Next, we have to specify a solution concept for~\eqref{rpde}. We rely on the mild formulation of~\eqref{rpde}, namely
\begin{equation}
\label{mild}
Y_{t} = S_{t}\xi + \int\limits_{0}^{t}S_{t-s}F(Y_{s})~\txtd s + \int\limits_{0}^{t}S_{t-s}G(Y_{s})~\txtd \textbf{W}_s.
\end{equation}
In order to give a meaning to~\eqref{mild} we introduce in the following sequel some concepts and notations from rough paths theory~\cite{FritzHairer}.\\ 
{\bf Notations.} We always let $|\cdot|_{\gamma}$ denote the H\"older-norm of $W$, $|\cdot|_{2\gamma}$ the $2\gamma$-H\"older norm of $\mathbb{W}$, and $\|\cdot\|_{\gamma,\alpha}$ stands for the $\gamma$-H\"older norm in $\cB_{\alpha}$. As a convention, the first index in $\|\cdot\|_{\gamma,\alpha}$ describes the time-regularity and the second one refers to the space regularity. Furthermore, the symbol $|\cdot|_{\gamma}$ always indicates the regularity of the random input, i.e.~we use the notation $|\cdot|$
even if we refer to the $\gamma$-H\"older norm of $W$ or to the $2\gamma$-H\"older norm of the second order process $\mathbb{W}$. The symbol $\|\cdot\|_{\gamma,\alpha}$ will be exclusively used to indicate time and space regularity. The notation $\|\cdot\|_{\infty,\alpha}$ stands for the supremum-norm in $\cB_{\alpha}$. Furthermore, $C$ stands for a universal constant which varies from line to line. We write $a \lesssim b$ if there exists a constant $C>0$ such that $a\leq C b$. The constant $C$ is allowed to depend on $F$, $G$ and their derivatives and on the parameters $\gamma$, $\alpha$ and $\rho_\gamma(\W)$,
but can be chosen uniformly on compact intervals. For our purposes we will state most of the estimates on the time interval $[0,1]$.\\ 

We now describe the space of the paths that can be integrated with respect to $\textbf{W}$ and observe that the setting here is different from the rough ODE case, where we showed that for a pair of controlled rough paths $(U,U')$ the convolution with the semigroup $(S_{t-\cdot}U, S_{t-\cdot}U')$ remains again a controlled rough path, see \cite[Lemma~2.6.1]{KN}. In the finite dimensional case, this is possible due to the Lipschitz continuity in time of the semigroup. However, this property does not hold true in infinite dimensions because the semigroup is not H\"older continuous in zero. More precisely, due to~\eqref{hg:2} we have for $0<s\leq t \leq T$ that
\begin{align}
\|S_t\xi - S_{s} \xi\|_{\cB} =\|(S_{t-s}-\mbox{Id})S_s\xi\|_{\cB} &\leq C (t-s)^{2\gamma}\|\xi\|_{\cB_{2\gamma}}\label{s:hoelder}\\
& \leq  C(t-s)^{2\gamma} s^{-2\gamma}\|\xi\|_{\cB},\label{m:hoelder}
\end{align}
which indicates a singularity in zero for initial data $\xi\in\cB$. Due to this fact, it is a challenge task to find the right function space of controlled rough paths in which to set up a fixed-point argument to solve rough PDEs. 
Several approaches have been considered in the literature to overcome this obstacle. For instance:
\begin{itemize}
	\item [(P1)] take more regular initial data, i.e.~$\xi\in \cB_{2\gamma}$;
	\item [(P2)] solve~\eqref{rpde} as in~\cite{HesseNeamtu1} in a modified H\"older space which compensates the {\em time}-singularity occurring in~\eqref{m:hoelder}, e.g. $$C^{2\gamma}_{2\gamma}([0,T];\cB):=\left\{\sup\limits_{0<s<t\leq T}s^{2\gamma}\frac{\|U_{s,t}\|_\cB}{(t-s)^{2\gamma}}<\infty\right\};$$
	\item [(P3)] require higher {\em space} regularity and solve~\eqref{rpde} in a larger space containing $\cB$ and use regularizing properties of analytic semigroups to show that the solution actually belongs to $\cB$ as in~\cite{GHairer,GHN}.
\end{itemize}

	
Regarding this we introduce the following definition of a controlled rough path tailored to the parabolic structure of the  PDE we consider. For other approaches see~\cite{GHairer} and~\cite[Section 12.2.2]{FritzHairer}.
\begin{definition} {\em (Controlled rough path according to a monotone family $(\cB_\alpha)_{\alpha\in \R}$).}\label{def:crp}
	We call a pair $(U,U')$ a controlled rough path if
	\begin{itemize}
		\item $(U,U')\in C([0,T];\cB_\alpha) \times ((C[0,T];\cB_{\alpha-\gamma} ) \cap C^{\gamma}([0,T];\cB_{\alpha-2\gamma} ))$. The component $U'$ is referred to as the Gubinelli derivative\footnote{For smooth paths $U$ and $W$, the choice of $U'$ is not unique. However, one can show that for rough inputs $W$, $U'$ is uniquely determined by $U$, see~\cite[Remark~4.7 and Section~6.2]{FritzHairer}.} of $U$.
		\item the remainder  \begin{align}\label{remainder}
		R^U_{s,t}= U_{s,t} -U'_s W_{s,t}    
		\end{align}
		belongs to $ C^{\gamma}([0,T];\cB_{\alpha-\gamma})\cap C^{2\gamma}([0,T];\cB_{\alpha-2\gamma})$.
	\end{itemize}
\end{definition}

The space of controlled rough paths is denoted by $D^{2\gamma}_{W,\alpha}$ and endowed with the norm $\|\cdot\|_{W,2\gamma,\alpha}$ given by
\begin{align}\label{g:norm}
\gubnorm{U}{U'}= \left\|U \right\|_{\infty,\cB_\alpha} 
+ \|U' \|_{\infty,\cB_{\alpha-\gamma}}
+ \left\|U'\right\|_{\gamma,\cB_{\alpha-2\gamma}}
+ \left\|R^U \right\|_{2\gamma,\cB_{\alpha-2\gamma}}.
\end{align}
 In order to emphasize the time horizon we write $D^{2\gamma}_{W,\alpha}([0,T])$ instead of $D^{2\gamma}_{W,\alpha}$.
\begin{remark}
	\begin{itemize}
		\item [1).] Note that we do not make the H\"older continuity of $U$ part of the definition of a controlled rough path, since
		using~\eqref{remainder} one immediately obtains for $\theta \in \left\{\gamma,2\gamma \right\}$ that
		\begin{align}\label{est:hoelder:y}
		\left\|U \right\|_{\gamma,\cB_{\alpha-\theta}}
		\leq \left\|U' \right\|_{\infty,\cB_{\alpha-\theta}} \left\|W\right\|_{\gamma} + \left\|R^y \right\|_{\gamma,\cB_{\alpha-\theta}}.
		\end{align}
	\item[2).] One can show that on a monotone scale of interpolation spaces the norm in~\eqref{g:norm} is equivalent to the apparently stronger one introduced in~\cite{GHN} which additionally includes $\|R^U\|_{\gamma,\alpha-\gamma}$. A proof of this statement can be found in~\cite{HocquetN} and relies on the interpolation inequality~\eqref{interpolation:ineq}.
			\item [3).] Definition~\ref{def:crp} states that $(U,U')\in D^{2\alpha}_{W,\gamma}$ is controlled by $W$ according to the monotone family of interpolation spaces $(\cB_\alpha)_{\alpha\in\mathbb{R}}$ as in~\cite{GHN}. One can make the semigroup $(S_t)_{t\geq 0}$ part of the definition of the controlled rough path as in~\cite{GHairer}. We work with Definition~\ref{def:crp}, since it reflects the appropriate space-time regularity of the solution and stays closer to the finite-dimensional setting~\cite{FritzHairer, Gubinelli}. Moreover, for the existence of center manifolds we will apply a cut-off technique to~\eqref{rpde}. For this argument it is also convenient not to incorporate the semigroup in the definition of the controlled rough path.
	\end{itemize}
\end{remark}

Throughout this section $(U,U')$ is going to denote an arbitrary controlled rough path and $(Y,Y')$ is used to refer to the solution of~\eqref{rpde}. Given a controlled rough path, one can introduce the rough integral as follows.

\begin{theorem}
\label{integral} Let $(U,U')\in D^{2\gamma}_{W,\alpha}$. Then 
\begin{align}\label{Gintegral}
\int\limits_{s}^{t} S_{t-r}U_{r}~\txtd \textbf{W}_{r} :=\lim\limits_{|\mathcal{P}|\to 0} \sum\limits_{[u,v]\in\mathcal{P}} S_{t-u}U_{u}W_{u,v} + S_{t-u}U'_{u}\mathbb{W}_{u,v},
\end{align}
where $\mathcal{P}$ denotes a partition of $[s,t]$. For $0\leq \beta< 3\gamma$ the following estimate
\begin{align}
\label{estimate:integral}
\Bigg\| \int\limits_{s}^{t} S_{t-r} U_{r}~\txtd\textbf{W}_{r} - S_{t-s}U_{s}W_{s,t} -S_{t-s}U'_{s}\mathbb{W}_{s,t} \Bigg\|_{\cB_{\alpha-2\gamma+\beta}} \lesssim \rho_\gamma(\W) \gubnorm{U}{U'} (t-s)^{3\gamma-\beta}
\end{align}
holds true. 
\end{theorem}

We emphasize that the stochastic convolution increases the spatial regularity of the controlled rough path, see~\cite[Corollary 4.6]{GHN} and Lemma 3.5 in~\cite{HN21}. We recall this result, which will be used later on.
\begin{corollary}\label{cor:higherreg}
	Let $(U,U')\in\cD^{2\gamma}_{W,\alpha}$, $T\in[0,1]$ and $0\leq \sigma<\gamma$. Then the integral map
	\begin{align*}
	(U,U')\mapsto (Z,Z') := \Big( \int_0^\cdot S_{\cdot-r}U_r~\txtd \W_r, U_\cdot \Big)
	\end{align*}
	maps $\cD^{2\gamma}_{W,\alpha}$ into itself. Moreover
	\begin{align*}
 \gubnormpar{Z}{Z'}{\alpha+\sigma} \leq |U_0|_{\alpha} + |U'_0|_{\alpha-\gamma} + C T^{\gamma-\sigma} (1+\rho_\gamma(\W)) \gubnorm{U}{U'}.
	\end{align*}

\end{corollary}

\textbf{Assumptions on drift and diffusion coefficients}:
\begin{itemize}
	\item [(\textbf{G})] 
	Let $\theta\in\{0,\gamma,2\gamma\}$ and $0\leq\sigma<\gamma$. The nonlinear diffusion coefficient $G:\cB_{\alpha-\theta}\to\cL(\R,\cB_{\alpha-\theta-\sigma})$ is Lipschitz and three times Fr\`echet  differentiable with bounded derivatives, i.e. $\|D^k G\|_{\cL(\cB^{\otimes k}_{\alpha-\theta},\cB_{\alpha-\theta-\sigma}) } <\infty$ for $k\in\{1,2,3\}$. Furthermore, we assume that
	\begin{align}
	\label{eq:dyncond1}
	G(0)=\txtD G(0)=\txtD^2 G(0)=0.
	\end{align}
In order to ensure global-in-time existence of solutions of~\eqref{rpde} we additionally assume that the derivative of
	$$\txtD G(\cdot)G(\cdot):\cB_{\alpha-\gamma}\to \cB_{\alpha-2\gamma-\sigma} $$ is bounded. This condition is satisfied in particular if $G$ itself is bounded or linear~\cite{HesseNeamtu2,HN21}.
	\item [(\textbf{F})] The drift term $F:\cB_{\alpha}\to \cB_{\alpha-\delta}$ for $\delta\in[0,1)$ is Lipschitz continuous. Furthermore, we assume that 	
	\be
	\label{eq:dyncond2}
	F(0)=\txtD F(0)=0.
	\ee
\end{itemize}

The conditions \eqref{eq:dyncond1}-\eqref{eq:dyncond2} guarantee that our main rough PDE~\eqref{rpde} has a steady state at $0\in \cB$; of course, up to a translation, we could take this point anywhere in $\cB$ but we fix it at $0$ for convenience. The global Lipschitz assumptions on the drift and diffusion coefficients can be weakened. However, for our setting, the above assumption suffices, since we will truncate the nonlinear terms in a neighbourhood of the origin as illustrated in Section~\ref{sect:est:sol}. Our main goal is to show that under these assumptions the rough PDE~\eqref{rpde} has a local center manifold in $\cB_\alpha$ for small initial data belonging to $\cB_\alpha$.


\section{The truncated rough PDE}
\label{sect:est:sol}


The next step is to modify $F$ and $G$ such that their Lipschitz constants become small. More precisely, for a fixed $R>0$ we need to compose $F$ and $G$ with a smooth cut-off function $\chi_{R}$. Here $R$ denotes the size of the ball around zero and in our case it will depend on the size of the noise, i.e.~$R=R\Big(|W|_{\gamma,[0,1]},|\mathbb{W}|_{2\gamma,[0,1]^2}\Big)$. Since we develop only a local theory, the exact size of this random ball is not crucial, since this can be chosen small enough as required in the fixed-point argument. In this setting, we recall that a composition of a controlled rough path with a smooth function is a well-defined operation~\cite[Lemma~7.3]{FritzHairer} for rough ODEs and~\cite[Lemma 4.7]{GHN} for rough PDEs. \\
The main novelty here is to define the composition of a controlled rough path $(U,U')\in\cD$ with a smooth cut-off function. Due to this procedure we obtain a rough PDE with path dependent coefficients. Therefore we have to show by means of fixed-point arguments that such an equation is well-posed.\\

For notational simplicity we set $\cD:=D^{2\gamma}_{W,\alpha}$ and define
\begin{align*}
\chi(U):=\begin{cases} U, &\mbox{ if  } \|(U,U')\|_{\mathcal{D}}\leq 1/2,\\
0, &\mbox{ if  } \|(U,U')\|_{\mathcal{D}}\geq 1.
\end{cases}
\end{align*}
For instance $\chi$ can be obtained as
\begin{align*}
\chi(U) = U f(\|U,U'\|_{\mathcal{D}}),
\end{align*}
where $f:\mathbb{R}^{+}\to[0,1]$ is a three-times continuously differentiable cut-off function with bounded derivatives, see~\cite{KN} for particular examples. In this case one has
\begin{align*}
(\chi (U), \chi (U)' ) =\begin{cases}
(U,U'), &\mbox{ if } \|(U,U')\|_{\cD} \leq 1/2\\
0, &\mbox{ if } \|(U,U')\|_{\cD} \geq 1.
\end{cases}
\end{align*}
Furthermore, for $R>0$ we set 
\begin{align*}
\chi_{R}(U) = R~\chi\Big(\frac{1}{R} U\Big)
\end{align*}
such that 
\begin{align*}
(\chi_{R} (U), \chi_{R} (U)' ) =\begin{cases}
 (U,U'), &\mbox{ if } \|(U,U')\|_{\cD} \leq R/2\\
0, &\mbox{ if } \|(U,U')\|_{\cD} \geq R.
\end{cases}
\end{align*}
We denote $F_{R}:=F\circ \chi_{R}$, $G_{R}:=G\circ \chi_{R}$ and have
\begin{align*}
F_{R}(U) = F (U) \mbox{ and } G_{R}(U)=G(U), ~\mbox{ if }~ \|U,U'\|_{\cD} \leq R/2.
\end{align*}
We set for $t\in [0,1]$:
$$\overline{F}(U)(t): = F(U_t) \mbox { and } \overline{G}(U)(t):=G(U_t)$$
and introduce the truncation as
$$ F_R(U):= \overline{F}\circ \chi_{R}(U) \mbox{ respectively } G_R:=\overline{G}\circ \chi_{R}(U).$$
This means that we have $$F_{R}(U)(t) = \overline{F}(\chi_{R}(U))(t)= F(\chi_{R}(U)_t )= F(U_t f (\|U,U'\|/R))$$
and analogously, $$G_{R}(U)(t) = \overline{G}(\chi_{R}(U))(t)= G(\chi_{R}(U)_t  ) = G(U_tf(\|U,U'\|/R) ),$$
where we removed for simplicity the index $\cD$ from $\|\cdot,\cdot\|_{\cD}$. 
The Gubinelli derivative of $G_R$ can be computed according to the chain rule~\cite[Lem.~7.3]{FritzHairer} as
\begin{align*}
(G_{R}(U))' = \txtD G (\chi_{R}(U)) (\chi_{R}(U))' = \txtD G (Uf (\|U,U'\|/R)) U' f(\|U,U'\|/R),
\end{align*}
since $f$ is a constant with respect to time.\\
Evaluating $(G_R(U))'$ at a time $t\in[0,1]$, we obtain
\begin{align*}
(G_{R}(U)(t))' = \txtD G (\chi_{R}(U)_t) (\chi_{R}(U)_t)' = \txtD G (U_tf (\|U,U'\|/R)) U'_t f(\|U,U'\|/R).
\end{align*}
By the definition of $\chi_{R}$ we have that $F_{R}(U)=\overline{F}(U)$ and $G_{R}(U)=\overline{G}(U)$ if $\|(U,U')\|_{\cD}\leq R/2$.
With this notation, the first component of the mild solution of~\eqref{rpde} equivalently rewrites as
\begin{align}\label{equiv:rde}
Y_t = S_t\xi + \int\limits_{0}^{t}S_{t-r}\overline{F}(Y)(r)~\txtd r + \int\limits_{0}^{t}S_{t-r}\overline{G}(Y)(r)~\txtd \textbf{W}_r. 
\end{align}
We now argue in several steps that the modified rough PDE~\eqref{rpde} has a unique solution in the space of controlled rough paths. Note that the coefficients $F_R$ and $G_R$ are now path dependent. We show that $F_R$ and $G_R$ are Lipschitz continuous with Lipschitz constants $L_F(R)$ and $L_G(R)$ such that $L_F(R)\to 0$ respectively $L_G(R)\to 0$ as $R\to 0$. Recall that $\cD=D^{2\gamma}_{W,\alpha}$ and that the universal constant $C$ is allowed to depend on $F$, $G$ and their derivatives. 
\begin{lemma}\label{drift:cutoff}
	Let $(U,U'), (\tilde{U},\tilde{U}')\in\cD$. Then there exists a constant $L_F(R)=L_F[R,F,\chi]$ such that $L_F(R)\to 0$ as $R\to 0 $ and that
	\begin{align}\label{fr}
	\Big\|\Big(\int_0^\cdot S_{\cdot-r}(F_{R}(U)(r) - F_{R}(\widetilde{U})(r)))~\txtd r,0\Big)\Big\|_{\cD}  &\leq L_F(R) \|U-\widetilde{U}, U'-\widetilde{U}'\|_{\cD}.
	\end{align}
\end{lemma}
\begin{proof}
	Recalling~\eqref{g:norm} and the fact that the Gubinelli derivative of the deterministic integral is zero we have to estimate $\|\int_0^\cdot S_{\cdot-r}(F_R(U)(r)-F_R(\tilde{U}(r)))~\txtd r\|_{\infty,\alpha}$ and the $2\gamma$-norm of the remainder of this convolution in $\cB_{\alpha-2\gamma}$. For the latter we compute for $0\leq s \leq t\leq 1$
	\begin{align*}
&	\int_0^t S_{t-r}[F_R(U)(r)-F_R(\tilde{U})(r)]~\txtd r - 	\int_0^s S_{s-r}[F_R(U)(r)-F_R(\tilde{U})(r)]~\txtd r \\
&= (S_{t-s}-\text{Id}) \int_0^s S_{s-r}[ F_R(U)(r) - F_R(\tilde{U})(r)]~\txtd r + \int_s^t S_{t-r} [F_R(U)(r) - F_R(\tilde{U})(r)]~\txtd r.
	\end{align*}
Let $i=0,1,2$ and recall that $F:\cB_\alpha\to\cB_{\alpha-\delta}$. The first term entails
\begin{align}
&\Bigg\| (S_{t-s}-\text{Id}) \int_0^s S_{s-r}[ F_R(U)(r) - F_R(\tilde{U})(r)]~\txtd r\Bigg\|_{\cB_{\alpha-i\gamma}} \nonumber\\
&\leq \| S_{t-s} -\text{Id}\|_{\cL(\cB_\alpha,\cB_{\alpha-i\gamma})} \int_0^s \|S_{s-r}\|_{\cL(\cB_{\alpha-\delta},\cB_\alpha)} \| [ F_R(U)(r) - F_R(\tilde{U})(r)] \|_{\cB_{\alpha-\delta}}~\txtd r\nonumber\\
& \leq (t-s)^{i\gamma} \int_0^s (s-r)^{-\delta} \| [ F_R(U)(r) - F_R(\tilde{U})(r)] \|_{\cB_{\alpha-\delta}}~\txtd r\nonumber\\
& \leq (t-s)^{i\gamma} s^{1-\delta} \sup\limits_{r\in[0,1]} \| [ F_R(U)(r) - F_R(\tilde{U})(r)] \|_{\cB_{\alpha-\delta}}\label{f1}.
\end{align} 
Analogously we have for the second term
\begin{align}
&\int_s^t \|S_{t-r} [F_R(U)(r) - F_R(\tilde{U})(r)]\|_{\alpha-i\gamma}~\txtd r \nonumber\\ &\leq \int_s^t \|S_{t-r}\|_{\cL(\cB_{\alpha-\delta},\cB_{\alpha-i\gamma})} \| [ F_R(U)(r) - F_R(\tilde{U})(r)] \|_{\cB_{\alpha-\delta}}~\txtd r\nonumber\\
& \leq \int_s^t (t-r)^{i\gamma-\delta} \| [ F_R(U)(r) - F_R(\tilde{U})(r)] \|_{\cB_{\alpha-\delta}}~\txtd r \nonumber  \\
	& \leq (t-s)^{\min\{ 1, 1+i\gamma-\delta \}} \sup\limits_{r\in[0,1]} \| [ F_R(U)(r) - F_R(\tilde{U})(r)] \|_{\cB_{\alpha-\delta}}\label{f2}.
\end{align}
Therefore, we compute $\sup\limits_{r\in[0,1]} \| [ F_R(U)(r) - F_R(\tilde{U})(r)] \|_{\cB_{\alpha-\delta}}$.  We have

	\begin{align*}
\sup\limits_{r\in[0,1]} \|F_R(U)(r) - F_{R}(\widetilde{U})(r)\|_{\cB_{\alpha-\delta}} &=\sup\limits_{r\in[0,1]} \|F(\chi_{R}(U)_r) - F(\chi_{R}(\widetilde{U})_r) \|_{\cB_{\alpha-\delta}}.
\end{align*}
This further results in 
\begin{align}
&\|F(\chi_{R}(U)_r) - F(\chi_{R}(\widetilde{U})_r) \|_{\cB_{\alpha-\delta}}	\leq \int\limits_{0}^{1} \|\txtD F(\tau\chi_{R}(U)_r + (1-\tau)\chi_R(\widetilde{U})_t)\|_{\cB_{\alpha-\delta}}~\txtd \tau~ \|\chi_{R}(U)_r -\chi_{R}(\widetilde{U})_r\|_{\cB_{\alpha-\delta}}\nonumber\\
&\leq \| \txtD F\|_{\cL(\cB_{\alpha},\cB_{\alpha-\delta})} \max \{\|\chi_{R}(U)_r\|_{\cB_\alpha}, \|\chi_{R}(\widetilde{U})_r\|_{\cB_\alpha}\} \|\chi_R(U)_r - \chi_R(\widetilde{U})_r\|_{\cB_{\alpha-\delta}}\nonumber\\
& \leq C R \|\chi_R(U)_r - \chi_R(\widetilde{U})_r\|_{\cB_\alpha}.\label{here}
\end{align}

Here we use that $ \|\chi_{R}(U)_r\|_{\cB_\alpha} \leq \|\chi_{R}(U)\|_{\infty,\alpha} =\|U f(\|U,U'\|/R)\|_{\infty,\alpha}=\|U\|_{\infty,\alpha} f(\|U,U'\|/R)$ together with the fact that $\cB_{\alpha}\subset \cB_{\alpha-\delta}$ therefore $\|\cdot\|_{\cB_{\alpha-\delta}}\leq C\| \cdot \|_{\cB_\alpha}$. 
 
Furthermore
\begin{align*}
\|\chi_R(U)_r - \chi_R(\tilde{U})_r\|_{\cB_\alpha} & = \|U_r f(\|U,U'\|/R) - \tilde{U}_r f(\|\tilde{U},\tilde{U}'\|/R)\|_{\cB_\alpha}\\
& \leq \|U_r -\tilde{U}_r\|_{\cB_\alpha} f(\|U,U'\|/R) +\|\tilde{U}_r\|_{\cB_\alpha}~ |f(\|U,U'\|/R) - f(\|\tilde{U},\tilde{U}'\|/R) |\\
& \leq \|U-\tilde{U}\|_{\infty,\alpha} + R \|\txtD f\|_{\infty} (\|U,U'\|/R - \|\tilde{U},\tilde{U}'\|/R)\\
& \leq (1+\|\txtD f\|_{\infty})  \|U -\widetilde{U}, U'-\widetilde{U}'\|_{\cD}.
\end{align*}
Consequently, this further yields
\begin{align*}
\sup\limits_{r\in[0,1]} \|F_R(U)(r) - F_{R}(\widetilde{U})(r) \|_{\cB_{\alpha-\delta}}&\leq C R \|\chi_{R}(U) -\chi_{R}(\widetilde{U})\|_{\infty,\alpha} 
\leq C [F,\chi] ~R~ \|U -\widetilde{U}, U'-\widetilde{U}'\|_{\cD}.
\end{align*}
Combining this with~\eqref{f1} and~\eqref{f2} proves the statement. \qed
	\end{proof}

\begin{lemma}\label{diffusion:cutoff}
	Let $(U,U'), (\widetilde{U},\widetilde{U}')\in \cD$. Then there exists a constant $L_G(R)=L_G[R, G, \rho_\gamma(\W),\chi]$ such that $L_G(R)\to 0$ as $R\to 0$ and that
	\begin{align}\label{composition:r}
	\|G_{R}(U) - G_R(\widetilde{U}), (G_{R}(U)- G_R(\widetilde{U}))'\|_{{D^{2\gamma}_{W,\alpha-\sigma}}} \leq L_G (R) \|U-\widetilde{U}, U'-\widetilde{U}'\|_{\cD}.
		\end{align}
\end{lemma}
\begin{proof}
	Regarding our assumption ({\bf G}) we can view the Fr\'echet derivative $\txtD^k G$ as an element of $\cL(\cB^{\otimes k}_{\alpha-i\gamma}, \cB_{\alpha-i\gamma-\sigma})$ for $k=1,2,3$ and $i=0,1,2$.  \\
	Due to~\eqref{g:norm} we have to estimate
	\begin{align*}
	&\|G_R(U)-G_R(\tilde{U})\|_{\infty,\alpha-\sigma} + \| (G_R(U) -G_R(\tilde{U}))'\|_{\infty,\alpha-\gamma-\sigma}\\& +\| (G_R(U) -G_R(\tilde{U}))'\|_{\gamma,\alpha-2\gamma-\sigma} 
	+ \|R^{G_R(U)-G_R(\tilde{U})} \|_{2\gamma,\alpha-2\gamma-\sigma}.
	\end{align*}
	We begin with the first term and write
	\begin{align*}
	\| G_R(U)-G_R(\tilde{U})\|_{\infty,\alpha-\sigma} =\sup\limits_{t\in[0,1]} \|G_R(U)(t) -G_R(\tilde{U})(t)\|_{\cB_\alpha-\sigma} =\sup\limits_{t\in[0,1]} \|G(\chi_{R}(U)_t) - G(\chi_{R}(\widetilde{U})_t) \|_{\cB_{\alpha-\sigma}}.
	\end{align*}
	Analogously to Lemma~\ref{drift:cutoff} and regarding that $\txtD G(0)=0$ we have
	\begin{align}\label{g1}
	&\|G(\chi_{R}(U)_r) - G(\chi_{R}(\widetilde{U})_r) \|_{\cB_{\alpha-\sigma}}	\leq \int\limits_{0}^{1} \|\txtD G(\tau\chi_{R}(U)_r + (1-\tau)\chi_R(\widetilde{U})_t)\|_{\cB_{\alpha-\sigma}}~\txtd \tau~ \|\chi_{R}(U)_r -\chi_{R}(\widetilde{U})_r\|_{\cB_{\alpha-\sigma}}\nonumber\\
	&\leq \| \txtD G\|_{\cL(\cB_{\alpha},\cB_{\alpha-\sigma})} \max \{\|\chi_{R}(U)_r\|_{\cB_\alpha}, \|\chi_{R}(\widetilde{U})_r\|_{\cB_\alpha}\} \|\chi_R(U)_r - \chi_R(\widetilde{U})_r\|_{\cB_{\alpha-\delta}}\nonumber\\
	& \leq C R \|\chi_R(U)_r - \chi_R(\widetilde{U})_r\|_{\cB_\alpha} \leq C R \|U-\tilde{U}\|_{\infty,\alpha} \leq CR \|U-\tilde{U},U'-\tilde{U}'\|_{\cD}.
	\end{align}
	For the terms containing the Gubinelli derivative we recall that for $t\in[0,1]$:
	\begin{align*}
	(G_{R}(U)(t))' = \txtD G (\chi_{R}(U)_t) (\chi_{R}(U)_t)' = \txtD G (U_tf (\|U,U'\|/R)) U'_t f(\|U,U'\|/R).
	\end{align*}
	
	We now investigate $\|(G_{R}(U) -G_R(\widetilde{U}))'\|_{\gamma,\alpha-2\gamma-\sigma}$. To this aim we let $0\leq s \leq t \leq 1$ and consider
	\begin{align*}
	&[G_R (U)(t) -G_R(\widetilde{U})(t) - (G_{R}(U)(s) -G_R(\widetilde{U})(s)) ]'\\
	& =[\txtD G(\chi_{R}(U)_t)(\chi_{R}(U)_t)' - \txtD G(\chi_R(\widetilde{U})_t)(\chi_R(\widetilde{U})_t)' - (\txtD G(\chi_R(U)_s ) (\chi_R(U)_s)' - \txtD G(\chi_R(\widetilde{U})_s)(\chi_R(\widetilde{U})_s)' )] \\
	& = \txtD G(U_t f(\|U,U'\|/R) )U'_t f(\|U,U'\|/R) - \txtD G(\widetilde{U}_t f (\|\widetilde{U},\widetilde{U}'\|/R) ) \widetilde{U}'_t  f (\|\widetilde{U},\widetilde{U}'\|/R) \\
	&   - (\txtD G(U_s f(\|U,U'\|/R) )U'_s f(\|U,U'\|/R) - \txtD G(\widetilde{U}_s f (\|U,U'\|/R) ) \widetilde{U}'_s  f (\|\widetilde{U},\widetilde{U}'\|/R)).
	\end{align*}
	Therefore, we have to estimate in $\cB_{\alpha-2\gamma-\sigma}$ the last expression. This further results in 
	\begin{align*}
	&\|[\txtD G(\chi_{R}(U)_t)(\chi_{R}(U)_t)' - \txtD G(\chi_R(\widetilde{U})_t)(\chi_R(\widetilde{U})_t)' - (\txtD G(\chi_R(U)_s ) (\chi_R(U)_s)' - \txtD G(\chi_R(\widetilde{U})_s)(\chi_R(\widetilde{U})_s)' )]\|\\
	& \leq \| \txtD G(\chi_R(U)_t) - \txtD G(\chi_R(U)_s) - \txtD G(\chi_R(\widetilde{U})_t) + \txtD G (\chi_R(\widetilde{U})_s) \| \\ &\hspace*{40 mm}\cdot\|(\chi_R(U)_t)' + (\chi_R(U)_s)' + (\chi_R(\widetilde{U}_t))' + (\chi_R(\widetilde{U})_s)' \|\\
	& ~~+ \| \txtD G(\chi_R(U)_t) - \txtD G(\chi_R(U)_s) + \txtD G(\chi_R(\widetilde{U})_t) - \txtD G (\chi_R(\widetilde{U})_s) \| \\& \hspace*{40 mm}\cdot\|(\chi_R(U)_t)' + (\chi_R(U)_s)' - (\chi_R(\widetilde{U}_t))' - (\chi_R(\widetilde{U})_s)' \|\\
	&~~ + \| \txtD G(\chi_R(U)_t) + \txtD G(\chi_R(U)_s) -\txtD G(\chi_R(\widetilde{U})_t) - \txtD G (\chi_R(\widetilde{U})_s) \| \\ &\hspace*{40 mm}\cdot\|(\chi_R(U)_t)' - (\chi_R(U)_s)' + (\chi_R(\widetilde{U}_t))' - (\chi_R(\widetilde{U})_s)' \|\\
	& ~~ + \| \txtD G(\chi_R(U)_t) + \txtD G(\chi_R(U)_s) + \txtD G(\chi_R(\widetilde{U})_t) + \txtD G (\chi_R(\widetilde{U})_s) \| \\&\hspace*{40 mm} \cdot\|(\chi_R(U)_t)' - (\chi_R(U)_s)' - (\chi_R(\widetilde{U}_t))' + (\chi_R(\widetilde{U})_s)' \|\\
	& := I + II + III +IV.
	\end{align*}
	
	We analyze each of the four terms above separately. For the first one we obtain 
	\begin{align*}
	\|I\|_{\cB_{\alpha-2\gamma-\sigma}}&\leq \|\txtD^2 G \|_{\cL(\cB^{\otimes 2}_{\cB_{\alpha-2\gamma},\cB_{\alpha-2\gamma-\sigma}})}R~ \|U-\widetilde{U}, U'-\widetilde{U}'\|_{\cD}~ [\|\chi_R(U)\|_{\infty,\alpha-2\gamma} +  \|\chi_R(\tilde{U})\|_{\infty,\alpha-2\gamma}]\\
	& \leq  C R^2 \|U-\widetilde{U}, U'-\widetilde{U}'\|_{\cD}.
	\end{align*}
	This estimate relies on applying~\eqref{g1} replacing $G$ by $\txtD G$ and regarding that $\txtD^2G(0)=0$. More precisely, we have
	\begin{align}
	\begin{split}\label{subg1}
	&| \txtD G(\chi_{R}(U)_t)  - \txtD G(\chi_{R}(\widetilde{U})_t) - [\txtD G(\chi_{R}(U)_s)  - \txtD G(\chi_{R}(\widetilde{U})_s)] |\\
	& \leq \|\txtD G\|_{\cL(\cB^{\otimes 2}_{\alpha-2\gamma}, \cB_{\alpha-2\gamma-\sigma} )} \max\{ \|\chi_{R}(U)_t\|, \|\chi_{R}(\widetilde{U})_t\|_{\cB_{\alpha-2\gamma}} \}~ \|\chi_{R}(U)_t - \chi_{R}(\widetilde{U})_t -  [ \chi_{R}(U)_s -\chi_{R}(\widetilde{U})_s ]\|_{\cB_{\alpha-2\gamma}}  \\
	&\hspace*{10 mm}+ C ~\|\chi_{R}(U)_t -\chi_{R}(\widetilde{U}_t)\|_{\alpha-2\gamma}~ [\|\chi_{R}(U)_t -\chi_{R}(U)_s\|_{\alpha-2\gamma} + \|\chi_{R}(\widetilde{U})_t -\chi_{R}(\widetilde{U})_s\|_{\alpha-2\gamma} ]\\
	\end{split}\\
	&\leq C R \|\chi_R (U) -\chi_{R}(\widetilde{U})\|_{\gamma,\alpha-2\gamma} + C \|\chi_{R}(U) - \chi_{U}(\widetilde{U})\|_{\infty,\alpha-2\gamma} ( \|\chi_{R}(U))\|_{\gamma,\alpha-2\gamma} +\|\chi_{R}(\widetilde{U})\|_{\gamma,\alpha-2\gamma} )\nonumber\\
	& \leq C R \|U-\widetilde{U}, U'-\widetilde{U}'\|_{\cD}. \nonumber
	\end{align}
Here we use	\eqref{est:hoelder:y} for $\theta=2\gamma$ to  control $\|\chi_R(U)\|_{\gamma,\alpha-2\gamma}$. Regarding these deliberations, we completed the estimate of $I$.\\
We now focus on $\|II\|_{\cB_{\alpha-2\gamma-\sigma}}$ which analogously results in 
	\begin{align*}
	\|II\|_{\cB_{\alpha-2\gamma-\sigma}}&\leq  C \|\txtD G(\chi_R(U)) -\txtD G(\chi_R(\widetilde{U}))\|_{\gamma, \alpha-2\gamma-\sigma} [\|\chi_R(U)\|_{\infty,\alpha-2\gamma} +\|\chi_R(\widetilde{U})\|_{\infty,\alpha-2\gamma} ]\\
	& \leq C R \|U - \widetilde{U}, U'-\widetilde{U}'\|_{\cD}.
	\end{align*} 
	Analogously, for the third one we have
	\begin{align*}
	\|III\|_{\cB_{\alpha-2\gamma-\sigma}}& \leq C \|\txtD G(\chi_R(U)) -\txtD G(\chi_R(\widetilde{U})) \|_{\infty,\alpha-2\gamma-\sigma} [\|\chi_R(U)\|_{\gamma,\alpha-2\gamma} +\|\chi_R(\widetilde{U})\|_{\gamma,\alpha-2\gamma}]\\
	& \leq C R \|U - \widetilde{U}, U'-\widetilde{U}'\|_{\cD}. 
	\end{align*}
	Finally, we easily estimate the fourth term regarding that $\txtD^2 G(0)=0$ as
	\begin{align*}
	\|IV\|_{\cB_{\alpha-2\gamma-\sigma}} &\leq C \sup\limits_{t\in[0,1]}\max\{\|\txtD G (\chi_R(U)_t)\|_{\alpha-2\gamma-\sigma},\|\txtD G(\chi_R(\widetilde{U})_t)\|_{\alpha-2\gamma-\sigma} \} \\ & \hspace*{80 mm}\cdot\|(\chi_R(U))' - (\chi_R(\widetilde{U}))'\|_{\gamma,\alpha-2\gamma}\\
	& \leq C R \|U - \widetilde{U}, U'-\widetilde{U}'\|_{\cD}.
	\end{align*}

	We now focus on the remainder. Recalling that $(U,U')\in \cD$, we know that
	$$U_t = U_s + U'_s W_{s,t} + R^{U}_{s,t}.$$ Therefore,
	$$U_t f(\|U,U'\|/R) = U_s f(\|U,U'\|/R) + U'_s f(\|U,U'\|/R) W_{s,t} + R^{U}_{s,t}f(\|U,U'\|/R),$$
	which gives us $R^{\chi_{R}(U)} = R^{U} f(\|Y,Y'\|/R)$.
	Consequently,
	\begin{align}\label{rem:cutoff}
	(R^{G_R(U)})_{s,t}= G(\chi_{R}(U)_t) - G(\chi_R(U)_s) - 
	\txtD G(\chi_{R}(U)_s)(\chi_{R}(U))_{s,t} + \txtD G (\chi_{R}(U)_s)R^{U}_{s,t} f(\|U,U'\|/R) .
	\end{align}
	This means that we have to estimate 
	\begin{align}
	R ^{G_{R}(U)}_{s,t} - R^{G_{R}(\widetilde{U})}_{s,t} &= G(\chi_{R}(U)_t) - G(\chi_R(U)_s) - 
	\txtD G(\chi_{R}(U)_s)(\chi_{R}(U))'_{s,t}\nonumber \\&- [G(\chi_{R}(\widetilde{U})_t) - G(\chi_R(\widetilde{U})_s) - 
	\txtD G(\chi_{R}(\widetilde{U})_s)(\chi_{R}(\widetilde{U}))'_{s,t}]\nonumber\\
	& + \txtD G (\chi_{R}(U)_s)R^{U}_{s,t} f(\|U,U'\|/R)  - \txtD G (\chi_{R}(\widetilde{U})_s)R^{\widetilde{U}}_{s,t} f(\|\widetilde{U},\widetilde{U}'\|/R). \label{last:rem}
	\end{align}
	The terms in~\eqref{last:rem} easily result in 
	\begin{align*}
	&\|\txtD G (\chi_R(U)_s) R^{\chi_R(U)}_{s,t} -\txtD G (\chi_R(\widetilde{U})_s) R^{\chi_R(\widetilde{U})}_{s,t}\|_{\cB_{\alpha-2\gamma-\sigma}} \\
	& \leq \|\txtD G (\chi_R(U)_s) [R^{\chi_R}(U)_{s,t} - R^{\chi_R(\widetilde{U})}_{s,t} ] \|_{\cB_{\alpha-2\gamma-\sigma}}+ \|[\txtD G (\chi_R(U)_s) -\txtD G(\chi_R(\widetilde{U})_s)] R^{\chi_R(\widetilde{U})}_{s,t}\|_{\cB_{\alpha-2\gamma-\sigma}}\\
	& \leq C \| \txtD G\|_{\cL(\cB_{\alpha-2\gamma},\cB_{\alpha-2\gamma-\sigma})}  \|\chi_R(U)_s\|_{\cB_{\alpha-2\gamma}} \| R^{\chi_R}(U) - R^{\chi_R(\widetilde{U})}\|_{2\gamma,\alpha-2\gamma}\\& + \|\txtD^2 G \|_{\cL(\cB^{\otimes 2}_{\alpha-2\gamma},\cB_{\alpha-2\gamma-\sigma})}  \|R^{\chi_R(\widetilde{U})}\|_{2\gamma,\alpha-2\gamma} \|\chi_R(U)_s -\chi_R(\widetilde{U})_s\|_{\alpha-2\gamma}\\
	& \leq C R \|R^{U} -R^{\widetilde{U}}\|_{2\gamma,\alpha-2\gamma} + C R \|U - \widetilde{U}\|_{\infty,\alpha-2\gamma}\\
	& \leq C R \|U-\tilde{U},U'-\tilde{U}'\|_{\cD}.
	\end{align*}
	To estimate the $2\gamma$-H\"older norm in $\cB_{\alpha-2\gamma}$ of the expressions~\eqref{last:rem} appearing in the difference of two remainders, we firstly recall the identity
	\begin{align}
&	G(U_t) - G(U_s) - \txtD G(U_s)U_{s,t} =\int\limits_{0}^{1} \txtD G (r U_t +(1-r) U_s )~\txtd r~ U_{s,t} - \txtD G(U_s)U_{s,t} \label{dg} \\
	& = \int\limits_{0}^{1} [\txtD G (r U_t +(1-r) U_s ) - \txtD G(U_s)] U_{s,t}~\txtd r\nonumber\\
	& = \int\limits_{0}^{1} \int\limits_{0}^1 \txtD^2 U [\tau (r U_t +(1-r) U_s)+(1-\tau)U_s ] (rU_t+(1-r)U_s-U_s) U_{s,t}~\txtd\tau~ \txtd r\nonumber\\
	&=\int\limits_{0}^1\int\limits_{0}^1 r \txtD^2 G [\tau (r U_t +(1-r) U_s)+(1-\tau)U_s ]~\txtd\tau~\txtd r ~(U_{s,t}\otimes U_{s,t}).\nonumber
	\end{align}
	Applying~\eqref{dg} twice, one obtains the following inequality (see p.~2716 in~\cite{HuNualart})
	\begin{align}
	&\|G(\chi_{R}(U)_t) - G(\chi_R(U)_s) - \txtD G(\chi_{R}(U)_s)(\chi_{R}(U))_{s,t} \\
	&- (G(\chi_{R}(\widetilde{U})_t) - G(\chi_R(\widetilde{U})_s) - DG(\chi_{R}(\widetilde{U})_s)(\chi_{R}(\widetilde{U}))_{s,t}) \|_{\cB_{\alpha-2\gamma-\sigma}} \nonumber\\
	\begin{split}
	& \leq \|\txtD^2 G\|_{\cL(\cB^{\otimes 2}_{\alpha-2\gamma,\cB_{\alpha-2\gamma-\sigma})}}\|\chi_{R}(U)_t\|_{\cB_{\alpha-2\gamma}} ~[\|\chi_{R}(U)_t -\chi_{R}(U)_s\|_{\cB_{\alpha-2\gamma}} +\|\chi_{R}(\widetilde{U})_t -\chi_{R}(\widetilde{U})_s \|_{\cB_{\alpha-2\gamma}} ]\\ & \hspace*{70 mm}\cdot\|\chi_{R}(Y)_t -\chi_{R}(Y)_s - (\chi_{R}(\widetilde{Y})_t -\chi_{R}(\widetilde{Y})_s)\|_{\cB_{\alpha-2\gamma}}\label{2}\\
	&+ \|\txtD^3 G\|_{\cL(\cB^{\otimes 3}_{\alpha-2\gamma,\cB_{\alpha-2\gamma-\sigma})}} \|\chi_{R}(\widetilde{U})_t -\chi_{R}(\widetilde{U})_s\|^2_{\cB_{\alpha-2\gamma}} ~[\|\chi_{R}(U)_t -\chi_{R}(\widetilde{U})_t \|_{\cB_{\alpha-2\gamma}} + \|\chi_{R}(U)_s -\chi_{R}(\widetilde{U})_s \|_{\cB_{\alpha-2\gamma}}].
	\end{split}
	\end{align}
	Using that $\| \cdot\|_{\alpha-2\gamma} \leq C \|\cdot\|_{\alpha-\gamma}\leq C \|\cdot\|_\alpha$, the previous inequality further leads to 
	\begin{align*}
	&\|G(\chi_{R}(U)_t) - G(\chi_R(U)_s) - \txtD G(\chi_{R}(U)_s)(\chi_{R}(U))_{s,t}\\& - (G(\chi_{R}(\widetilde{U})_t) - G(\chi_R(\widetilde{U})_s) - \txtD G(\chi_{R}(\widetilde{U})_s)(\chi_{R}(\widetilde{U}))_{s,t}) \|_{\cB_{\alpha-2\gamma-\sigma}}\\
	& \leq C R ~ [\|U\|_{\gamma,\alpha-\gamma} + \|\widetilde{U}\|_{\gamma,\alpha-\gamma}]~ \|U-\widetilde{U}\|_{\gamma,\alpha-\gamma} + C\|\chi_R(\widetilde{U})\|^{2}_{\gamma,\alpha-2\gamma} \|U-\widetilde{U}\|_{\infty,\alpha} \\
	&\leq  C R^2 \|U-\widetilde{U}, U'- \widetilde{U}'\|_{\cD}.
	\end{align*} 
	Here we used that $(\chi_R(\tilde{U}), \chi_R(\tilde{U})')\in\cD$ and~\eqref{est:hoelder:y} for $\theta=2\gamma$
	 in order to estimate the second term of the previous inequality as follows
	\begin{align*}
	\|\chi_R(\tilde{U})\|_{\gamma,\alpha-2\gamma}& \leq \|\chi_R(\tilde{U})'\|_{\infty,\alpha-2\gamma} \|W\|_{\gamma}+ \|R^{\chi_R(\tilde{U})}\|_{\gamma,\alpha-2\gamma} \\
	& \leq \|\chi_R(\tilde{U})'\|_{\infty,\alpha-\gamma} \|W\|_{\gamma}+ \|R^{\chi_R(\tilde{U})}\|_{\gamma,\alpha-2\gamma} \leq C (1+\|W\|_{\gamma}) \|\chi_R(\tilde{U}),\chi_R(\tilde{U})'\|_{\cD} \leq C R.
	\end{align*}	
	Putting all these estimates together proves the statement.
	\end{proof}\\

 We now replace $F$ and $G$ in~\eqref{rpde} by $F_R$ and $G_{R}$ and show that the modified rough PDE obtained by this procedure has a unique solution in $\cD$. For notational simplicity we introduce
 for $(Y,Y')\in \cD$ and $t\in[0,1]$
 \begin{align}\label{tr}
 T_{R}({W}, Y,Y')[t]:&=  \int\limits_{0}^{t} S_{t-r} 
 \overline{F}(\chi_R(Y))(r) ~\txtd r + \int\limits_{0}^{t} S_{t-r} \overline{G}(\chi_R(Y))(r) 
 ~\txtd \bm{{W}}_{r}\\
 & = \int\limits_{0}^{t} S_{t-r} 
 F_R(Y)(r) ~\txtd r + \int\limits_{0}^{t} S_{t-r} G_R(Y)(r) 
 ~\txtd \bm{{W}}_{r} \nonumber,
 \end{align}
 with Gubinelli derivative $(T_{R}(W,Y,Y'))'=G_{R}(Y)$.
 Collecting the estimates derived in Lemmas~\ref{drift:cutoff} and~\ref{diffusion:cutoff} leads to the following result.
 
\begin{theorem}\label{fpr}
	There exists a unique solution $(Y,Y')\in \cD$ of~\eqref{rpde} satisfying $Y'=G_R(Y)$ such that for $t\in[0,1]$:
	\begin{align*}
	Y_t = S_t \xi + \int\limits_{0}^{t} S_{t-r} 
	F_R(Y)(r) ~\txtd r + \int\limits_{0}^{t} S_{t-r} G_R(Y)(r) 
	~\txtd \bm{{W}}_{r}.
	\end{align*} 
\end{theorem}
\begin{proof}
	The proof relies on Banach's fixed-point theorem and the main novelty here is to incorporate the path-dependence of $F_R$ and $G_R$ which is possible due to Lemma~\ref{drift:cutoff} and~\ref{diffusion:cutoff}.
	For the term containing the initial data we have that $(S_t\xi,0)\in\cD$ since
	\begin{align*}
	\|S_t\xi\|_{\cB_\alpha} \leq \|S_t\|_{\cL(\cB_\alpha,\cB_\alpha)}\|\xi\|_{\cB_\alpha} \leq C \|\xi\|_{\cB_\alpha}
	\end{align*}
	and  
	\begin{align*}
	\|R^{S_\cdot\xi}\|_{\cB_{\alpha-2\gamma}} &=\| S_t \xi - S_s \xi\|_{\cB_{\alpha-2\gamma}} = \|(S_{t-s}-\text{Id})S_s\xi\|_{\cB_{\alpha-2\gamma}} \\
	& \leq \|S_{t-s}-\text{Id}\|_{\cL(\cB_\alpha,\cB_{\alpha-2\gamma})} \|S_s\xi\|_{\cB_\alpha}\\
	& \leq C (t-s)^{2\gamma} \|\xi\|_{\cB_\alpha}.
	\end{align*}
		Let $(Y,Y'), (\widetilde{Y},\widetilde{Y}')\in \cD $ with $Y_0=\widetilde{Y}_0=\xi$ and $Y'_0=\tilde{Y}'_0$.
By Lemma~\ref{drift:cutoff} we have that
\begin{align*}
	\left\|\int\limits_{0}^{\cdot} S_{\cdot-r} (F_R(Y)(r) - F_R(\widetilde{Y})(r) )~\txtd r, 0 \right\|_{\cD} \leq L_F(R) \|Y-\tilde{Y}\|_{\cD}.
\end{align*}
Furthermore, Corollary~\ref{cor:higherreg} combined with Lemma~\ref{diffusion:cutoff} entails
\begin{align*}
	\left \| \int\limits_{0}^{\cdot} S_{\cdot - r} 
	(G_R(Y)(r) -G_R(\widetilde{Y})(r )) ~\txtd \W_{r}, G_R(Y)-G_R(\widetilde{Y})\right\|_{\cD} 
\lesssim {\rho_\gamma(\W)} L_G(R) \|Y-\widetilde{Y}, Y'-\widetilde{Y}'\|_{\cD}.
\end{align*}
In conclusion we obtain
\begin{align}
	&	\left\|\int\limits_{0}^{\cdot} S_{\cdot-r} (F_R(Y)(r) - F_R(\widetilde{Y})(r) )~\txtd r+ \int\limits_{0}^{\cdot} S_{\cdot - r} 
	(G_R(Y)(r) -G_R(\widetilde{Y})(r )) ~\txtd \W_{r}, 
	G_R(Y) - G_R(\widetilde{Y}) \right\|_{\cD}\nonumber\\ 
	&\lesssim {\rho_\gamma(\W)}  \Big(L_R(F) + L_G(R) \Big) \|Y-\widetilde{Y}, Y'-\widetilde{Y}'\|_{\cD}.
\end{align}
		
			Setting $\widetilde{Y}\equiv 0$ and using that $F_R(0)=G_R(0)=0$,  we see from the previous deliberations that $(Y,Y')\mapsto (T_R(W,Y,Y'), G_R(Y))$ maps $\cD$ into itself. Furthermore, by choosing $R$ small enough and regarding that $L_F(R)\to 0$ and $L_G(R)\to 0$ as $R\to 0$, we obtain that the mapping $(Y,Y')\mapsto (T_R(W,Y,Y'), G_R(Y))$ is a contraction. Therefore, Banach's fixed-point theorem proves the statement.
	\qed
	\end{proof}
\begin{remark}
	\begin{itemize}
		\item [1).] This argument is slightly different from the proof of~\cite[Theorem 5.1]{GHN} where one takes $\gamma'$ such that $0<\gamma<\gamma'<\sigma$ and performs the fixed-point argument in $\cD^{2\gamma'}_{W,\alpha}$ choosing the time horizon small enough. 
		\item [2).] Due to the assumptions ({\bf F}) and ({\bf G}) the solution of the rough PDE~\eqref{rpde} (and particularly of truncated one) exists globally in time~\cite{HN21}. 
	\end{itemize}

\end{remark}

Going back to our setting, the next aim is to characterize the parameter $R$ required in order to
decrease the Lipschitz constants of $\overline{F}$ and $\overline{G}$ using $\chi_R$. This fact is required for the Lyapunov-Perron method. As already seen we have to choose $R$ as small as possible. Since in our deliberations, it is always required that $R\leq 1$ and $L_F(R)\to 0$, $L_G(R)\to 0$ as $R\to 0$,
collecting the previous estimates and regarding the structure of the constants $L_F(\cdot)$ and $L_G(\cdot)$  entails
\begin{align}
&	\left\|\int\limits_{0}^{\cdot} S_{\cdot-r} (F_R(Y)(r) - F_R(\widetilde{Y})(r) )~\txtd r+ \int\limits_{0}^{\cdot} S_{\cdot - r} 
(G_R(Y)(r) -G_R(\widetilde{Y})(r )) ~\txtd \W_{r}, 
G_R(Y) - G_R(\widetilde{Y}) \right\|_{\cD}\nonumber\\ 
&\leq  \Big(C_F R + C_G \widetilde{C}[\rho_\gamma(\W)] R\Big) \|Y-\widetilde{Y}, Y'-\widetilde{Y}'\|_{\cD},\label{e:finall}
\end{align}
for constants $C_F>0$ and $C_G>0$ which depend on $F$, $G$, their derivatives and on the cuf-off function $\chi$. Furthermore the constant $\widetilde{C}>0$ incorporates the dependence of the random input contained by $\rho_\gamma(\W)$.
As commonly met in the theory of random dynamical systems~\cite{LuSchmalfuss, GarridoLuSchmalfuss}, since all the estimates depend on the random input, it is meaningful to employ a cut-off technique for a {\em random variable}, i.e.~$R=R(W)$. Such an argument will also be used here as follows.\\

We fix $K>0$ and regarding~\eqref{e:finall}, we let 
$\widetilde{R}({W})$ be the unique solution of 
\begin{equation}\label{k}
(C_F + C_G \widetilde{C}[\rho_\gamma(\W)])~\widetilde{R}({W}) =K
\end{equation}
and set 
\begin{align}\label{r}
R({W}):=\min\{\widetilde{R}({W}),1\}.
\end{align}
This means that if $R(W)=1$, we apply the cut-off procedure for $\|Y,Y'\|_{\cD}\leq 1/2$ or else for $\|Y,Y'\|_{\cD}\leq R(W)/2$. In conclusion, we work in the next sections with a modified version of the rough PDE~\eqref{rpde} (equivalently~\eqref{equiv:rde}), where the drift and diffusion coefficients $\overline{F}$ and $\overline{G}$ are replaced by $F_{R(W)}$ and $G_{R(W)}$. For notational simplicity, the $W$-dependence of $R$ will be dropped whenever there is no confusion.\\

Due to~\eqref{k} we conclude:
\begin{lemma}  
	Let $(Y,Y'), (\widetilde{Y},\widetilde{Y}')\in \cD$. We have
	\begin{equation}
	\label{wanttohave}
	\|T_{R}({W},Y,Y') - T_{R}({W},\widetilde{Y}, \widetilde{Y}'), (T_{R}({W},Y,Y') - T_{R}({W},\widetilde{Y}, \widetilde{Y}'))'
	\|_{\cD} \leq K \|  Y- \widetilde{Y}, 
	Y'-\widetilde{Y}'\|_{\cD}.
	\end{equation}
\end{lemma}

\begin{remark}
	\begin{itemize}
		\item [1)] We emphasize that one needs to control the derivatives of the diffusion coefficient $G$ in order to make the constants that depend on $G$ small after the cut-off procedure. Such restrictions are often met in the context of invariant manifolds for stochastic partial differential equations with nonlinear multiplicative noise, see e.g.~\cite{FuBloemker,GarridoLuSchmalfuss}.
		\item [2)]	If the random input is smoother, i.e.~$\gamma\in(1/2,1)$, it is enough to assume only $\txtD G(0)=0$. In this case, the stochastic convolution used in~\eqref{Gintegral} is defined as a Young integral.
	\end{itemize}
\end{remark}

\section{Random Dynamical Systems}
\label{sect:rds}

The main techniques and results established in the previous section using controlled rough paths are necessary in order to provide pathwise estimates for the solutions of~\eqref{rpde}. In this section, we provide some concepts from the random dynamical systems theory~\cite{Arnold}, which allow us to define a center manifold for~\eqref{rpde}.\medskip

The next concept is fundamental in the theory of random dynamical systems, since it describes a model of the driving noise.

\begin{definition} 
Let $(\Omega,\mathcal{F},\mathbb{P})$ be a probability space and $\theta:\mathbb{R}\times\Omega\rightarrow\Omega$ be a family of $\mathbb{P}$-preserving transformations (i.e.,~$\theta_{t}\mathbb{P}=\mathbb{P}$ for $t\in\mathbb{R}$) having the following properties:
\begin{itemize}
	\item[(i)] the mapping $(t,\omega)\mapsto\theta_{t}\omega$ is 
		$(\mathcal{B}(\mathbb{R})\otimes\mathcal{F},\mathcal{F})$-measurable, where 
		$\mathcal{B}(\cdot)$ denotes the Borel sigma-algebra;
		\item[(ii)] $\theta_{0}=\textnormal{Id}_{\Omega}$;
		\item[(iii)] $\theta_{t+s}=\theta_{t}\circ\theta_{s}$ for all 
		$t,s,\in\mathbb{R}$.
\end{itemize}
Then the quadrupel $(\Omega,\mathcal{F},\mathbb{P},(\theta_{t})_{t\in\mathbb{R}})$ is called a metric dynamical system.
\end{definition}

In this framework we recall that we consider one-dimensional noise, since the generalization to higher dimensions does not require any additional arguments. In our context, constructing a metric dynamical system is going to rely on constructing $\theta$ as a shift map on a canonical probability space $(\Omega,\cF,\P)$ as specified below.  Recalling that $\gamma\leq 1/2$ was fixed in Section~\ref{preliminaries}, we define for an $\gamma$-H\"older rough path $\textbf{W}=(W,\mathbb{W})$ and $\tau\in\mathbb{R}$ the time-shift $\Theta_{\tau}\textbf{W}:=(\Theta_{\tau}W,\Theta_{\tau}\mathbb{W})$ by
\begin{align*}
& \Theta_{\tau} W_t : =W_{t+\tau} - W_{\tau}\\
& \Theta_{\tau}\mathbb{W}_{s,t}:=\mathbb{W}_{s+\tau,t+\tau}.
\end{align*}
Note that the time shift naturally extends linearly to sums of rough paths, e.g., $\Theta_{\tau} W_{s,t}=W_{t+\tau} - W_{s+\tau}$. Furthermore, the shift leaves the path space invariant:

\begin{lemma}
\label{shift} 
Let $T_{1}, T_{2},\tau\in\mathbb{R}$, and $\bm{W}=(W,\mathbb{W})$ be an $\gamma$-H\"older rough path on $[T_{1},T_{2}]$ for $\gamma\in(1/3,1/2)$. Then the time-shift $\Theta_{\tau}\bm{W}= (\Theta_{\tau}W, \Theta_{\tau}\mathbb{W})$ is also an $\gamma$-H\"older rough path on $[T_{1}-\tau, T_{2}-\tau]$.
\end{lemma}

\begin{proof}
Let $s,u,t\in [T_{1}-\tau, T_{2}-\tau]$. The $\gamma$-H\"older-continuity of $\theta_{\tau}W$ and the $2\gamma$-H\"older continuity of $\theta_{\tau}\mathbb{W}$ are obvious. We only prove that Chen's relation~\eqref{chen} holds true. We have
\begin{align}
	\Theta_{\tau} \mathbb{W}_{s,t} - \Theta_{\tau} \mathbb{W}_{s,u} 
	-\Theta_{\tau}\mathbb{W}_{u,t} &= \mathbb{W}_{s+\tau,t+\tau} 
	- \mathbb{W}_{s+\tau,u+\tau} - \mathbb{W}_{u+\tau,t+\tau} \nonumber \\
	& = W_{s+\tau,u+\tau} \otimes W_{u+\tau,t+\tau} \label{chena},\\
	&=(W_{u+\tau} -W_{\tau} - W_{s+\tau} +W_{\tau} ) \otimes (W_{t+\tau} 
	- W_{\tau} -W_{u+\tau} + W_{\tau} ) \nonumber \\
	& = \Theta_{\tau} W_{s,u} \otimes \Theta_{\tau} W_{u,t}. \nonumber
\end{align}
where in~\eqref{chena} we use Chen's relation~\eqref{chen}.\qed
\end{proof}

Based upon~\cite{BailleulRiedelScheutzow} we consider the following concept:

\begin{definition}
	\label{rpc}
	Let $(\Omega,\mathcal{F},\mathbb{P},(\theta_{t})_{t\in\mathbb{R}})$ be a 
	metric dynamical system. We call $\textbf{W}=(W,\mathbb{W})$ a rough path 
	cocycle if the identity
	\begin{align*}
	\textbf{W}_{s,s+t}(\omega)=\textbf{W}_{0,t}(\theta_{s}\omega)
	\end{align*}
	holds true for every $\omega\in\Omega$, $s\in\mathbb{R}$ and $t\geq 0$.
\end{definition}

The previous definitions hint already at the fact that one may be able to just use as a probability space $\Omega$ a space of paths. A classical case, where we get via this construction a metric dynamical 
system and a rough cocycle is the fractional Brownian motion, see also~\cite[Section~6]{HesseNeamtu2}.

\begin{example}
\label{eq:Bmain}
As a concrete example for $W$ consider the fractional Brownian motion $B^{H}$ restricted to any compact interval $[-L,L]$ with $L\geq 1$ and for $H\in(1/3,1/2]$. This includes classical Brownian motion as the case when $H=1/2$. Then $B^H$ can be lifted to a $\gamma$-H\"older rough-path $\textbf{B}^{H}=(B^{H}, \mathbb{B}^{H})$ as discussed in \cite[Example.~10.11]{FritzHairer}, where
\begin{align*}
\mathbb{B}^{H}_{s,t}:= \int\limits_{s}^{t} B^{H}_{s,u} \otimes ~\txtd B^{H}_{u}.
\end{align*} 	
Gluing together lifts on compact time intervals, one may extend $\textbf{B}^{H}$ to the whole real line. Furthermore, we may consider the canonical probability space $(C_{0}(\mathbb{R}),\mathcal{B}(C_{0}(\mathbb{R})),\mathbb{P} )$, where $C_{0}(\mathbb{R})$ denotes the space of all real-valued continuous functions, which are $0$ in $0$, endowed with the compact open topology. The shift on the sample path space is given by
\begin{align}
	(\Theta_{\tau} f) (\cdot) := f(\tau+\cdot) - f(\tau),  
	~ \tau\in\mathbb{R}, \quad f\in C_{0}(\R).
\end{align}
Using Kolmogorov's Theorem or the Garsia-Rodemich-Rumsey inequality~\cite[A.2]{FritzVictoir} one can conclude that maps in $C^{\gamma}_{0}(\mathbb{R})$ have a finite $\gamma$-H\"older semi-norm on every compact interval $\mathbb{P}$-almost surely. Hence, we can restrict this metric dynamical system to the set $C^{\gamma}_{0}(\mathbb{R})$. For the metric dynamical system 
\benn
	(C_{0}^\gamma(\mathbb{R}),\mathcal{B}(C_{0}^\gamma(\mathbb{R})),\mathbb{P},(\Theta_{t})_{t\in\R} )=:(\Omega_B,\cF_B,\mathbb{P},
	(\Theta_{t})_{t\in\mathbb{R}})
\eenn
one may check that $\textbf{B}^{H}=(B^{H},\mathbb{B}^{H})$ represents a rough path cocycle as introduced in Definition~\ref{rpc}. 
\end{example}

Of course, virtually the identical construction of a path-space $(\Omega_W,\cF_W,\mathbb{P})$, also referred to as canonical probability space, can be carried out for more general $\gamma$-H\"older rough paths $\textbf{W}=(W,\mathbb{W})$ constructed from a (stochastic) process $W$, not just fractional 
Brownian motion, where the definition of a shift map is still as above, i.e., 
\benn
(\Theta_{\tau} W )(t) : =W_{t+\tau} - W_{\tau}.
\eenn
We now have the abstract definition of, as well as concrete examples for, metric dynamical systems for our problem modeling the underlying rough driving process. Now we have to also define the dynamical systems structure of the solution operator of our rough stochastic PDE. As a first step we recall the definition of a random dynamical system (RDS)~\cite{Arnold} and show that the solution operator of~\eqref{rpde} generates an RDS in $\cB_\alpha$.

\begin{definition}
\label{rds} 
A random dynamical system on a separable Banach space $\cX$ over a metric dynamical system $(\Omega,\mathcal{F},\mathbb{P},(\theta_{t})_{t\in\mathbb{R}})$ is a mapping $$\varphi:[0,\I)\times\Omega\times \cX\to \cX,
	\mbox{  } (t,\omega,x)\mapsto \varphi(t,\omega,x), $$
	which is $(\mathcal{B}([0,\I))\otimes\mathcal{F}\otimes
	\mathcal{B}(\cX),\mathcal{B}(\cX))$-measurable and satisfies:
	\begin{description}
		\item[(i)] $\varphi(0,\omega,\cdot{})=\textnormal{Id}_{\cX}$ 
		for all $\omega\in\Omega$;
		\item[(ii)]$ \varphi(t+\tau,\omega,x)=
		\varphi(t,\theta_{\tau}\omega,\varphi(\tau,\omega,x)), 
		\mbox{ for all } x\in \cX, ~t,\tau\in[0,\I),~\omega\in\Omega;$
		\item[(iii)] $\varphi(t,\omega,\cdot{}):\cX\to \cX$ is 
		continuous for all $t\in[0,\I)$ and all $\omega\in\Omega$.
	\end{description}
\end{definition}

The second property in Definition~\ref{rds} is referred to as the cocycle property.~In order to investigate random dynamical systems for~\eqref{rpde} we need the global-in-time well-posedness of~\eqref{rpde}, which is guaranteed by our assumptions ({\bf F}) and ({\bf G})~\cite{HesseNeamtu2,HN21}. Moreover working with a pathwise interpretation of the stochastic integral as given in~\eqref{Gintegral}, no exceptional sets can occur. Therefore one can immediately infer that the solution operator of~\eqref{rpde} generates a RDS. For completeness, we sketch a proof of this fact, see also~\cite[Theorem 6.5]{HesseNeamtu2}. 

\begin{lemma}
\label{cocycle} 
Let $\xi\in\cB_\alpha$ and $\textbf{W}=(W,\WW)$ be a rough path cocycle. Then the solution operator of the rough PDE~\eqref{rpde}
\begin{align*}
	t\mapsto \varphi(t,W,\xi)=Y_{t}= S_t\xi + \int\limits_{0}^{t}S_{t-r}F(Y_{r})~\txtd r 
	+ \int\limits_{0}^{t} S_{t-r}G(Y_{r})~\txtd \bm{{W}}_{r}, 
\end{align*}
generates a random dynamical system in $\cB_\alpha$ over the metric dynamical system $(\Omega_W,\cF_W,\mathbb{P},(\Theta_t)_{t\in\R})$.
\end{lemma}

\begin{proof}
	The relevant properties to define the metric dynamical system we need
	have been discussed in Example~\ref{eq:Bmain}. The cocycle property can be immediately verified, since
	\begin{align*}
	Y_{t+\tau} &=S_{t+\tau}\xi + \int\limits_{0}^{t+\tau} S_{t+\tau-r} 
	F(Y_{r}) ~\txtd r + \int\limits_{0}^{t+\tau} S_{t+\tau-r} 
	G(Y_{r}) ~\txtd \bm{{W}}_{r}\\
	& = S_t S_\tau \xi + \int\limits_{0}^{\tau}  S_{t+\tau- r} 
	F(Y_{r}) ~\txtd r + \int\limits_{\tau}^{t+\tau} S_{t+\tau-r} 
	F(Y_{r}) ~\txtd r\\
	& + \int\limits_{0}^{\tau}S_{t+\tau-r} G(Y_{r}) ~\txtd 
	\bm{{W}}_{r} + \int\limits_{\tau}^{t+\tau} S_{t+\tau-r} G(Y_{r}) 
	~\txtd \bm{{W}}_{r}\\
	&=S_t \left( S_{\tau}\xi + \int\limits_{0}^{\tau} S_{\tau-r} 
	F(Y_{r}) ~\txtd r + \int\limits_{0}^{\tau} S_{\tau-r} ~\txtd 
	\bm{{W}}_{r}  \right)\\
	& + \int\limits_{0}^{t}  S_{t-r} F(Y_{r+\tau}) ~\txtd r + 
	\int\limits_{0}^{t} S_{t-r} G(Y_{r+\tau}) ~\txtd \Theta_{\tau}
	\bm{{W}}_{r}\\
	& = S_{t} Y_{\tau} + \int\limits_{0}^{t} S_{t-r} F(Y_{r+\tau}) 
	~\txtd r + \int\limits_{0}^{t}S_{t-r}G(Y_{r+\tau}) ~\txtd 
	\Theta_{\tau}\bm{{W}}_{r}.
	\end{align*}
The above computations are rigorously justified, since one can immediately check the shift property of the rough integral~\eqref{Gintegral}. For a complete proof of this statement, see \cite[Corollary 4.5]{HesseNeamtu1}. The $(\mathcal{B}([0,\infty))\otimes\mathcal{F}_{W}\otimes\mathcal{B}(\cB_\alpha), \mathcal{B}(\cB_\alpha))$-measurability of $\varphi$ follows by well-known arguments. One considers a sequence of (classical) solutions $(Y^{n},(Y^{n})')_{n\in\mathbb{N}}$ of~\eqref{rpde} corresponding to smooth approximations $(W^{n},\mathbb{W}^{n})_{n\in\mathbb{N}}$ of $(W,\mathbb{W})$. Obviously, the mapping $(t,W,\xi)\mapsto Y^{n}_{t}$ is $(\mathcal{B}([0,T])\otimes\mathcal{F}_{W}\otimes\mathcal{B}(\cB_\alpha), \mathcal{B}(\cB_\alpha))$-measurable for any $T>0$. Since $Y$ continuously depends on the rough input $W$, according to \cite[Lemma 3.12]{GHairer}, one immediately concludes that $\lim\limits_{n\to\infty}Y^{n}_t=Y_{t}$. This gives the measurability of $Y$ with respect to $\mathcal{F}_{W}\otimes \mathcal{B}(\cB_\alpha)$. Due to the time-continuity of $Y$, we obtain by~\cite[Chapter~3]{CastaingValadier} the $(\mathcal{B}([0,T])\otimes\mathcal{F}_{W}\otimes\mathcal{B}(\cB_\alpha), \mathcal{B}(\cB_\alpha))$-measurability of the mapping $(t,\omega,\xi)\mapsto Y_{t}$ for any $t\geq 0$.\qed
\end{proof}

\textbf{Notation:}~The role of the random elements in $\Omega_W$ is played by the paths $W$ as one uses the canonical probability space of paths. So we directly denote these elements by $W$ and do not write the identification $W_t(\omega):=\omega(t)$. However, this should be kept in mind. The random dynamical system $\varphi:\mathbb{R}^{+}\times \Omega\times \cB_\alpha \to \cB_\alpha$ obviously depends upon the $t,\xi,W$, and $\mathbb{W}$ although we do not directly display the dependence upon $\mathbb{W}$ in the notation.\medskip

To construct local random invariant manifolds, which can be characterized by the graph of a smooth function in a ball with a random radius~\cite{DuanLuSchmalfuss,GarridoLuSchmalfuss} one requires the concept of tempered random variables~\cite[Chapter 4]{Arnold}, which we recall next: 

\begin{definition}\label{t}
A random variable $R:\Omega\to (0,\infty)$ is called tempered from above, with respect to a metric dynamical system $(\Omega,\mathcal{F},\mathbb{P}, (\theta_{t})_{t\in\mathbb{R}})$, if
\begin{equation}
\label{tempered}
\limsup\limits_{t\to\pm\infty}\frac{\ln^{+} R (\theta_{t}\omega)}{t}=0, \quad \mbox{ for all } \omega\in\Omega,
\end{equation}	
where $\ln^+a:=\max\left\{\ln a,0\right\}$. A random variable is called tempered from below if $1/R$ is tempered from above. A random variable is tempered if and only if it is tempered from above and from below.
\end{definition}

Note that the set of all $\omega\in\Omega$ satisfying (\ref{tempered}) 
is invariant with respect to any shift map $(\theta_{t})_{t\in\mathbb{R}}$,
which is an observation applicable to our case when $\theta_t=\Theta_t$. A 
sufficient condition for temperedness from above is according to~\cite[Prop.~4.1.3]{Arnold} 
that
\begin{equation}
\mathbb{E} \sup\limits_{t\in[0,1]} R(\theta_{t}\omega)<\infty.
\end{equation}
Moreover, if the random variable $R$ is tempered from below with 
$t\mapsto R(\theta_{t}\omega)$ continuous for all $\omega\in\Omega$, 
then for every $\tilde\delta>0$ there exists a constant $C[\tilde\delta,\omega]>0$ 
such that
\begin{equation}
\label{temperedmanradius}
R(\theta_{t}\omega) \geq C[\tilde\delta,\omega] \txte^{-\tilde\delta|t|},
\end{equation}
for any $\omega\in\Omega$. Again, for our concrete example when $\Omega=\Omega_B$
one can easily check that norms are tempered.

\begin{lemma}
\label{ltempered} 
	Let $\textbf{B}^H=(B^H,\mathbb{B}^H)$ be the rough path cocycle associated to a fractional 
	Brownian motion $B^H$ with Hurst parameter $H\in(1/3,1/2]$. Then the random variables 
	\benn
	R_1(B^H)=\|B^H\|_{\gamma}\quad \text{ and }\quad R_2(\mathbb{B}^H)=\|\mathbb{B}^H \|_{2\gamma} 
	\eenn
	are tempered from above.
\end{lemma}

\begin{proof}
	The first assertion is valid due to the fact that 
	$\mathbb{E}\|B^{H}\|^{n}_{\gamma}<\infty$ and the second one follows regarding that
	$\mathbb{E}\|\mathbb{B}^{H}\|^{m}_{2\gamma}<\infty$, for $m\in\mathbb{N}$ 
	as contained in~\cite[Theorem~10.4]{FritzHairer}. This shows the temperedness from above
	of both random variables. \qed
\end{proof}

As before, the last result holds more generally for broader classes of Gaussian rough
paths, see~\cite[Section~10]{FritzHairer}. From now, we shall simply assume that 
$\textbf{W}=(W,\mathbb{W})$ is a  rough path cocycle such that the 
random variables
\benn
R_1(W)=\|W\|_{\gamma}\quad \text{ and }\quad R_2(\mathbb{W})=\|\mathbb{W} \|_{2\gamma}  
\eenn
are tempered from above. This concept is essential, since one wants to ensure that for initial conditions 
belonging to a ball with a sufficiently small tempered from below radius, the corresponding trajectories remain within such a ball, see~\cite{LuSchmalfuss,LianLu,CaraballoDuanLuSchmalfuss}. We state an easy fact explicitly, which is crucial in this context:

\begin{lemma}
\label{localcman}
The random variable $R$ in~\eqref{k} is tempered from below.
\end{lemma}

\begin{proof}
Using Lemma~\ref{ltempered} and~\ref{wanttohave} we immediately obtain that $\widetilde{C}[\rho_\gamma(\W)]$ is tempered from above. Recalling Definition~\ref{t}, we conclude that $\tilde{R}$ is tempered from below and therefore $R$ is also tempered from below.  \qed
\end{proof}

\section{Local Center Manifolds for Rough PDEs}
\label{lcm}

In this section we prove the existence of a local center manifold for~\eqref{rpde}. The technique is similar to the one employed in~\cite{KN}. The major technical difficulty is that we have to consider the fixed-point problem for the Lyapunov-Perron map in different function spaces. We state now the precise dynamical assumptions near the steady state at the origin, as shortly indicated in Section~\ref{heuristics}.
\begin{assumptions}\label{ass:linearpart} 
The spectrum of the linear operator $A$ is supposed to contain eigenvalues with zero and strictly negative real parts, i.e.~$\sigma(A)=\sigma^{\txtc}(A)\cup \sigma^{\txts}(A)$, where $\sigma^{\txtc}(A)=\{\lambda \in \sigma(A)\mbox{ : } \mbox{Re}(\lambda)=0\}$ and $\sigma^{\txts}(A)=\{\lambda\in \sigma(A) \mbox{ : } \mbox{Re}(\lambda)<0 \}$. The subspaces generated by the eigenvectors corresponding to these eigenvalues are denoted by $\cB^{\txtc}$ respectively $\cB^{\txts}$ and are referred to as {\em center} and {\em stable} subspace. These subspaces provide an invariant splitting of $\cB=\cB^{\txtc}\oplus \cB^{\txts}$.
We denote the restrictions of $A$ on $\cB^{\txtc}$ and $\cB^{\txts}$ by $A_{\txtc}:=A|_{\cB^{\txtc}}$ and $A_{\txts}:=A|_{\cB^{\txts}}$. Since $\cB^\txtc$ is finite-dimensional we obtain that $S^{\txtc}(t):=\txte^{tA_{\txtc}}$ is a {\em group} of linear operators on $\cB^{\txtc}$. Moreover, there exist projections $P^{\txtc}$ and $P^{\txts}$ such that $P^{\txtc} + P^{\txts} = \mbox{Id}_{\cB}$ and $A_{\txtc} =A|_{\cR(P^{\txtc})}$ and $A_{\txts}=A|_{\cR(P^{\txts})}$, where $\cR$ denotes the range of the corresponding projection.  In this case one can show that there is also a decomposition of the corresponding interpolation spaces $\cB_\alpha=\cB^\txtc \oplus \cB^\txts_\alpha$, where $\cB^\txts_\alpha =\cB_\alpha\cap \cB^\txts $~\cite{Simonett}. Additionally, we impose the following exponential dichotomy condition on the semigroup.
We assume that there exist two exponents $\gamma^*$ and $\beta^*$ with $-\beta^*<0\leq \gamma^*<\beta^*$ and constants $M_{\txtc},M_{\txts}\geq 1$, such that the following dichotomy condition is satisfied 
\begin{align*}
&\|S^{\txtc}(t)  x\|_{\cB} \leq M_{\txtc} \txte^{\gamma^* t} \|x\|_{\cB}, 
~~~\mbox{  for } t\leq 0 \mbox{ and } x\in \cB;\\
& \|S^{\txts}(t) x\|_{\cB} \leq M_{\txts} \txte^{-\beta^* t} \|x\|_{\cB}, 
~~\mbox{for } t\geq 0 \mbox{ and } x\in \cB.
\end{align*}
This yields according to~\cite[Theorem 2.1.3, p. 289]{Amann} a dichotomy condition also on the interpolation spaces $\cB_\alpha$ for $\alpha>0$, i.e.
\begin{align}
&\|S^{\txtc}(t)x\|_{\cB_\alpha} \leq M_{\txtc} \txte^{\gamma^* t} \|x\|_{\cB_\alpha}, 
~~~\mbox{  for } t\leq 0 \mbox{ and } x\in \cB_\alpha;\label{da}\\
& \|S^{\txts}(t) x\|_{\cB_\alpha} \leq M_{\txts} \txte^{-\beta^* t} \|x\|_{\cB_\alpha}, 
~~\mbox{for } t\geq 0 \mbox{ and } x\in \cB_\alpha. \label{sa}
\end{align}
For further details and similar assumptions ~\cite[Section~7.1, p.~460]{SellYou}. 
\end{assumptions}

\begin{remark} 
One can extend the techniques and results presented below easily if one additionally has an unstable subspace, namely if there exist eigenvalues of $A$ with real part greater than zero. In this case the classical exponential trichotomy condition is satisfied, see for instance~\cite[Section 7.1]{SellYou}. 
\end{remark}
\begin{definition}
We call a random set $\cM^{\txtc}({W})$, which is invariant 
with respect to $\varphi$ (i.e. $\varphi(t, {W},\cM^{\txtc}({W}))\subset 
\cM^{\txtc}(\Theta_{t}{W})$ for $t\in\mathbb{R}$ and ${W}\in\Omega_{W}$), a 
center manifold if this can be represented as 
\begin{align}\label{graph}
\cM^{\txtc}({W})=\{\xi + h^{\txtc}(\xi,{W})\mbox{ : }\xi\in \cB^{\txtc} \},
\end{align}
where $h^{\txtc}(\cdot,{W}):\cB^{\txtc}\to \cB^{\txts}_\alpha$ is Lipschitz continuous and differentiable in zero. Moreover, 
$h^{\txtc}(0,{W})=0$ and $\cM^{\txtc}({W})$ is tangent to $\cB^{\txtc}$ at 
the origin, meaning that the tangency condition $\txtD h^{\txtc}(0,{W})=0$ 
is satisfied.
\end{definition}
We show that~\eqref{rpde} has a local center manifold $\cM^\txtc (W)\subset \cB_\alpha$ for small initial data belonging to $\cB_\alpha$. Before constructing this local center manifold using the Lyapunov-Perron transform we further introduce the following notation. For $(U,U')\in \mathcal{D}$ we write:
\begin{align}\label{tt}
T^{\txts/\txtc}({W},U,U')[\cdot]:=  \left(\int\limits_{0}^{\cdot} S^{\txts/\txtc}_{\cdot-r} F(U) ~\txtd r + \int\limits_{0}^{\cdot} S^{\txts/\txtc}_{\cdot-r}G(U_{r}) ~\txtd \bm{{W}}_{r}  , G(U_{\cdot}) \right),
\end{align}
and 
\begin{align}
\label{hattt}
\hat{T}^{\txtc}({W}, U, U')[\cdot] := \left(\int\limits_{\cdot}^{1} 
S^{\txtc}_{\cdot- r } F(U_{r}) ~\txtd r + \int\limits_{\cdot} ^{1} 
S^{\txtc}_{\cdot -r } G(U_{r}) ~\txtd \bm{{W}}_{r}, G(U_{\cdot}) \right).
\end{align}
Given the spectral decomposition of $A$, the Lyapunov-Perron map for~\eqref{rpde} should be defined by, as discussed in Section~\ref{heuristics} and suppressing the dependence of $Y'$, as follows:
\begin{align}
\label{lp}
J({W}, Y)[\tau] & := S^{\txtc}_{\tau} \xi^{\txtc} + \int\limits_{0}^{\tau} 
S^{\txtc}_{\tau-r} F(Y_{r}) ~\txtd r + \int\limits_{0}^{\tau} S^{\txtc}_{\tau-r} G(Y_{r}) ~\txtd \textbf{W}_r\\
& +\int\limits_{-\infty}^{\tau} S^{\txts}_{\tau-r} F(Y_{r}) ~\txtd r 
+ \int\limits_{-\infty}^{\tau} S^{\txts}_{\tau-r} G (Y_{r}) 
~\txtd \textbf{W}_{r}\mbox{,  ~for } \tau\in\mathbb{R}_{-}. \nonumber
\end{align}
Since we are dealing with rough integrals and we have to control the H\"older norm of the noise on each time-interval, we have to appropriately discretize~\eqref{lp} as justified already for the rough ODE case in~\cite{KN}. Hence, we introduce a discrete version of the Lyapunov-Perron transform $J_{d}({W}, \mathbb{Y},\xi)$ for a sequence of controlled rough paths $\mathbb{Y}\in BC^{\eta}(\mathcal{D}) $ 
and $\xi\in \cB_\alpha$ as the pair $J_{d}({W},\mathbb{Y},\xi):=(J^{1}_{d}({W},\mathbb{Y},\xi), J^{2}_{d}({W},\mathbb{Y},\xi))$, where the precise structure is given below. For $t\in[0,1], ~{W}\in\Omega_{W} \mbox{ and } i\in\mathbb{Z}^{-}$ we define
\begin{align}
\label{j}
&J^{1}_{d}({W}, \mathbb{Y},\xi)[i-1,t] : = S^{\txtc}_{t+i-1} \xi^{\txtc} \\
&  -\sum\limits_{k=0}^{i+1} S^{\txtc}_{t+i-1-k} 
\left(\int\limits_{0}^{1} S^{\txtc}_{1-r} F_{R}(Y^{k-1}_{r}) ~\txtd r
+ \int\limits_{0}^{1} S^{\txtc}_{1-r}G_{R}(Y^{k-1}_{r}) ~\txtd \Theta_{k-1}
\bm{{W}}_{r}  \right)\nonumber\\
& - \int\limits_{t}^{1} S^{\txtc}_{t-r}F_{R}(Y^{i-1}_{r}) ~\txtd r 
- \int\limits_{t}^{1} S^{\txtc}_{t-r} G_{R}(Y^{i-1}_{r}) ~\txtd 
\Theta_{i-1} \bm{{W}}_{r}\nonumber\\
& + \sum\limits_{k=-\infty}^{i-1} S^{\txts}_{t+i-1-k} \left(
\int\limits_{0}^{1} S^{\txts}_{1-r} F_{R}(Y^{k-1}_{r}) ~\txtd r 
+ \int\limits_{0}^{1} S^{\txts}_{1-r}G_{R}(Y^{k-1}_{r}) ~\txtd \Theta_{k-1} 
\bm{{W}}_{r}  \right)\nonumber\\
& + \int\limits_{0}^{t} S^{\txts}_{t-r}F_{R}(Y^{i-1}_{r}) ~\txtd r
+ \int\limits_{0}^{t} S^{\txts}_{t-r} G_{R}(Y^{i-1}_{r}) ~\txtd 
\Theta_{i-1} \bm{{W}}_{r}. ~~ \nonumber
\end{align}
Furthermore, $J^{2}_{d}({W},\mathbb{Y},\xi)$ stands for the Gubinelli 
derivative of $J^{1}_{d}({W},\mathbb{Y},\xi)$, i.e.~$J^{2}_{d}({W},
\mathbb{Y},\xi)[i-1,\cdot]:=(J^{1}_{d}({W},\mathbb{Y},\xi)[i-1,\cdot] )'$. 
Note that $\xi^{\txtc}$ can be recovered setting $i=0$ and $t=1$ in the 
definition of $J^{1}_{d}({W},\mathbb{Y},\xi)$, i.e., $J^{1}_{d}({W},
\mathbb{Y},\xi)[-1,1]=\xi^{\txtc}$. The discretization of the Lyapunov-Perron map can be immediately derived using the substitution $\tau\mapsto t+i-1$ in~\eqref{lp} as computed in~\cite[Section~4.1]{KN}.
\medskip

We emphasize that for a sequence $\mathbb{Y}\in BC^{\eta}(\mathcal{D})$ 
the first index $i\in\mathbb{Z}_{-}$ in the definition of $J_{d}({W},
\mathbb{Y},\xi)[\cdot,\cdot]$ gives the position within the sequence and 
the second one refers to the time variable $t\in[0,1]$. Not to overburden 
the notation in~\eqref{j} for the elements of $\mathbb{Y}$ we simply 
write $Y^{i}_{t}$ instead of $Y[i,t]$ for $i\in\mathbb{Z}_{-}$ and $t\in[0,1]$.\medskip

\begin{remark}
	\begin{itemize}
		\item [1)] We are going to show that~\eqref{j} maps $BC^{\eta}(\mathcal{D})$ 
		into itself and is a contraction if the constant $K$ specified in~\eqref{k} is 
		chosen small enough.
		\item [2)] Compared to~\cite{KN}, several technical difficulties arise due to the fact that the controlled rough paths now incorporate different space and time regularity, recall Definition~\ref{def:crp}. Moreover, the dichotomy condition (\eqref{da} and \eqref{sa}) in the corresponding interpolation spaces is a crucial step for the following computation.
	\end{itemize}
\end{remark} 

We let $C_{S}$ stand for a 
constant which exclusively depends on the semigroup $S$ and derive:
\begin{theorem}
\label{thm:fp} Let Assumptions~\ref{ass:linearpart},
{\em ({\bf F}),~({\bf G})} hold and let $K$ satisfy the gap condition
\begin{align}
\label{gap:k}
	K C_S \left( \frac{ \txte^{\beta^*+\eta}(M_{\txts}
		\txte^{-\eta}+1)}{1-\txte^{-(\beta^*+\eta)}} + 
	\frac{\txte^{\gamma^*-\eta} (M_{\txtc}\txte^{-\eta}+1)}{1-
		\txte^{-(\gamma^*-\eta)}} \right) <\frac{1}{4}.
	\end{align}
	Then, the map $J_{d}:\Omega\times BC^{\eta}(\mathcal{D})\to 
	BC^{\eta}(\mathcal{D}) $ possesses a unique fixed-point 
	$\Gamma\in BC^{\eta}(\mathcal{D})$.
\end{theorem}

\begin{remark}
	Note that~\eqref{gap:k} can be obtained for instance by choosing 
	the constant appearing in~\eqref{k} as 
	\begin{align}\label{gk}
	K^{-1}:= 4 C_S \txte ^{(\beta^*+\gamma^*)/2}\left(  
	\frac{\txte^{(\beta^*-\gamma^*)/2} (M_{\txts}+M_{\txtc}) 
		+1} {1-\txte^{-(\beta^*+\gamma^*)/2}} \right),
	\end{align}
	which follows by setting $\eta:=\frac{-\beta^*+\gamma^*}{2}<0$.
\end{remark}

\begin{proof}
 Let two sequences $\mathbb{Y}=((Y^{i-1}, (Y^{i-1})'))_{i\in\mathbb{Z}^{-}}$ 
	and $\mathbb{\widetilde{Y}}=((\widetilde{Y}^{i-1}, 
	(\widetilde{Y}^{i-1})'))_{i\in\mathbb{Z}^{-}}$ belong to 
	$BC^{\eta}(\mathcal{D}) $ and satisfy $P^{\txtc}Y^{-1}_{1}=
	P^{\txtc}\widetilde{Y}^{-1}_{1}=\xi^{\txtc}$. We want to verify the contraction 
	property. The fact that $J_{d}(\cdot)$ maps $BC^{\eta}(\mathcal{D})$ 
	into itself can be derived by setting $\widetilde{\mathbb{Y}}=0$ in the 
	next computation and using that $F_{R}(0)=G_{R}(0)=0$. Keeping~\eqref{g:norm} in mind we compute as in the proof of Theorem~\ref{fpr} using~\eqref{da}
	\begin{align}
	\|S^{\txtc}_{t+i-1}\xi^{\txtc},0 \|_{BC^{\eta}(\mathcal{D})}& 
	= (\|S^\txtc_{\cdot+i+1}\xi^\txtc\|_{\infty,\alpha} + \|R^{S^\txtc_{\cdot+i+1}\xi^\txtc}\|_{2\gamma,\alpha-2\gamma} ) \txte^{-\eta (i-1)}\nonumber\\
	&\leq C_S\|S^\txtc_{i+1}\xi^\txtc\|_{\cB_\alpha} \txte^{-\eta(i-1)}\nonumber \\
	&\leq C_{S}M_{\txtc} \txte^{(\gamma^*-\eta)(i-1)}\|\xi^{\txtc}\|_{\cB_\alpha}.\label{last1}
	\end{align}
More precisely, the previous computation uses for $0\leq s\leq t\leq 1$ that
	\begin{align*}
	\|S^\txtc_{t+i-1}\xi^\txtc\|_{\cB_\alpha} \leq \|S^\txtc_t\|_{\cL(\cB_\alpha,\cB_\alpha)}\|S^{\txtc }_{i+1}\xi\|_{\cB_\alpha} \leq C_S e^{\gamma^*(i-1)} \|\xi^\txtc\|_{\cB_\alpha}
\end{align*}
and  
\begin{align*}
	\|R^{S^\txtc_{\cdot+i-1}\xi^\txtc}\|_{\cB_{\alpha-2\gamma}} &=\| S^\txtc_{t+i-1} \xi - S^\txtc_{s+i-1} \xi^\txtc\|_{\cB_{\alpha-2\gamma}} = \|(S^\txtc_{t-s}-\text{Id})S^\txtc_{s+i-1}\xi^\txtc\|_{\cB_{\alpha-2\gamma}} \\
	& \leq \|S^\txtc_{t-s}-\text{Id}\|_{\cL(\cB_\alpha,\cB_{\alpha-2\gamma})} \|S^\txtc_{s+i-1}\xi^\txtc\|_{\cB_\alpha}\\
	& \leq C_S e^{\gamma^*(i-1)} (t-s)^{2\gamma} \|\xi^\txtc\|_{\cB_\alpha}.
\end{align*}
 The expression~\eqref{last1} remains bounded for $i\in\mathbb{Z}^{-}$ since we assumed that $-\beta^*<\eta<0\leq\gamma^*<\beta^*$. Next, we are going to estimate the difference
	\begin{align*}
	||J_{d}({W},\mathbb{Y},\xi) - J_{d}({W},\widetilde{\mathbb{Y}},\xi)||_{BC^{\eta}(\mathcal{D})}
	\end{align*}
	in several intermediate steps. Verifying the contraction property on the 
	stable part of~\eqref{j}, one has to compute two terms. First of all, 
	due to~\eqref{wanttohave} we get
	\begin{align*}
	& \sum\limits_{k=-\infty}^{i-1} \txte^{-\eta(i-1)}
	\Big\|\Big(S^{\txts}_{\cdot+i-1-k} \Big(T^{\txts}_{R}(\Theta_{k-1}{W}, 
	Y^{k-1}, (Y^{k-1})') [1]- T^{\txts}_{R}(\Theta_{k-1}{W}, 
	\widetilde{Y}^{k-1}, (\widetilde{Y}^{k-1})')[1]\Big),0\Big)\Big\|_{\mathcal{D}}\\
	& \leq \sum\limits_{k=-\infty}^{i-1} C_{S} M_{\txts}\txte^{-\eta(i-1)} 
	\txte^{-\beta^*(i-1-k)}  K \|Y^{k-1}-\widetilde{Y}^{k-1}, 
	(Y^{k-1}-\widetilde{Y}^{k-1})'  \|_{\mathcal{D}}\\
	& = \sum\limits_{k=-\infty}^{i-1}  C_{S}M_{\txts} \txte^{-\eta(i-1)} 
	\txte^{-\beta^*(i-1-k)} \txte^{\eta(k-1)} K \txte^{-\eta(k-1)} 
	\| Y^{k-1}-\widetilde{Y}^{k-1},(Y^{k-1} -\widetilde{Y}^{k-1})'
	\|_{\mathcal{D}}\\
	& = \sum\limits_{k=-\infty}^{i-1}   \txte^{-(\eta+\beta^*) 
		(i-1-k) } C_{S} M_{\txts} \txte^{-\eta} K \txte^{-\eta(k-1)} \| 
	Y^{k-1}-\widetilde{Y}^{k-1},(Y^{k-1} 
	-\widetilde{Y}^{k-1})'\|_{\mathcal{D}}.
	\end{align*}
	For the first part of the computation, the only time-dependence is incorporated in $S^\txts_{\cdot+i-1-k}$ and the rough integrals appearing in $T^{\txts}_{R}$ are taken from zero to one and can be estimated using Lemma~\ref{diffusion:cutoff}. This entails
	\begin{align*}
		\|I_1\|_{\cB_\alpha}:& = \|T^{\txts}_{R}(\Theta_{k-1}{W}, 
		Y^{k-1}, (Y^{k-1})') [1]- T^{\txts}_{R}(\Theta_{k-1}{W}, 
		\widetilde{Y}^{k-1}, (\widetilde{Y}^{k-1})')[1]\|_{\cB_\alpha}\\ &\lesssim_{\rho_\gamma(\W)}  \| 
		Y^{k-1}-\widetilde{Y}^{k-1},(Y^{k-1} 
		-\widetilde{Y}^{k-1})'\|_{\mathcal{D}}.
	\end{align*}
	 Regarding the structure of the controlled rough path norm given by~\eqref{g:norm} we have to estimate the $\infty$-norm of $S^{\txts}_{\cdot+i-1-k} I_1$ in $\cB_\alpha$ and the $2\gamma$-H\"older norm of the remainder of this expression in $\cB_{\alpha-2\gamma}$.
	This gives us regarding~\eqref{sa}
	\begin{align*}
	&	\Big\|\Big(S^{\txts}_{\cdot+i-1-k} \Big(T^{\txts}_{R}(\Theta_{k-1}{W}, 
		Y^{k-1}, (Y^{k-1})') [1]- T^{\txts}_{R}(\Theta_{k-1}{W}, 
		\widetilde{Y}^{k-1}, (\widetilde{Y}^{k-1})')[1]\Big),0\Big)\Big\|_{\mathcal{D}}\\
		&  \leq (\sup\limits_{t\in[0,1]}\| S^\txts_{t+i-1-k}\|_{\cL(\cB_\alpha)} +  \sup\limits_{s\in[0,1]} \|S^\txts_{t-s}-\text{Id}\|_{\cL(\cB_{\alpha},\cB_{\alpha-2\gamma})} \| S^\txts_{s+i-1-k}\|_{\cL(\cB_\alpha)}  ) \|I_1\|_{\cB_\alpha}\\
		& \leq C_S M_\txts e^{-\beta^*(i-k-1)} \| 
		Y^{k-1}-\widetilde{Y}^{k-1},(Y^{k-1} 
		-\widetilde{Y}^{k-1})'\|_{\mathcal{D}}.
	\end{align*}
	Combining the previous computation with the last term of~\eqref{j}  entails the final 
	estimate on the stable part
	\begin{align*}
	& \sum\limits_{k=-\infty}^{i-1} \txte^{-\eta(i-1)}\Big\|\Big(S^{\txts}_{\cdot+i-1-k} \Big(T^{\txts}_{R}(\Theta_{k-1}{W}, 
	Y^{k-1}, (Y^{k-1})') [1]- T^{\txts}_{R}(\Theta_{k-1}{W}, 
	\widetilde{Y}^{k-1}, (\widetilde{Y}^{k-1})')[1]\Big),0\Big)\Big\|_{\mathcal{D}}\\
	& + \txte^{-\eta(i-1)} \|T^{\txts}_{R}(\Theta_{i-1}{W}, Y^{i-1}, 
	(Y^{i-1})') [\cdot]- T^{\txts}_{R}(\Theta_{i-1}{W}, \widetilde{Y}^{i-1}, 
	(\widetilde{Y}^{i-1})')[\cdot] \|_{\mathcal{D}}\\
	& \leq \sum\limits_{k=-\infty}^{i} \txte^{-(\eta+\beta^*)(i-k-1)} 
	K \widetilde{C}(M_{\txts}\txte^{-\eta} + 1)  e ^{-\eta(k-1)}  \| Y^{k-1}-
	\widetilde{Y}^{k-1},(Y^{k-1} -\widetilde{Y}^{k-1})'\|_{\mathcal{D}}\\
	& \leq 	K C_S  \frac{ \txte^{\beta^*+\eta}(M_{\txts}\txte^{-\eta}+1)}{1
		-\txte^{-(\beta^*+\eta)}}  \|\mathbb{Y}-\widetilde{\mathbb{Y}}, \mathbb{Y}'-\widetilde{\mathbb{Y}}'
	\|_{BC^{\eta}(\mathcal{D})}.
	\end{align*}
	We focus now on the center part. Here we obtain by the same arguments as above
	\begin{align*}
	& \sum\limits_{k=0}^{i+1} \txte^{-\eta(i-1)}\Big\|\Big(S^{\txtc}_{\cdot+i-1-k} \Big(T^{\txts}_{R}(\Theta_{k-1}{W}, 
	Y^{k-1}, (Y^{k-1})') [1]- T^{\txts}_{R}(\Theta_{k-1}{W}, 
	\widetilde{Y}^{k-1}, (\widetilde{Y}^{k-1})')[1]\Big),0\Big)\Big\|_{\mathcal{D}}\\
	& \leq \sum\limits_{k=0}^{i+1}  C_{S} M_{\txtc} \txte^{-\eta (i-1)} 
	\txte^{\gamma^*(i-1-k)} K \| Y^{k-1}-\widetilde{Y}^{k-1},(Y^{k-1} 
	-\widetilde{Y}^{k-1})'\|_{\mathcal{D}}\\
	& = \sum\limits_{k=0}^{i+1}  C_{S}M_{\txtc} \txte^{-\eta(i-1)} 
	\txte^{\gamma^*(i-1-k)} \txte^{\eta(k-1)} \txte^{-\eta(k-1)} 
	K \| Y^{k-1}-\widetilde{Y}^{k-1},(Y^{k-1} -\widetilde{Y}^{k-1})'
	\|_{\mathcal{D}}\\
	& = \sum\limits_{k=0}^{i+1}  C_{S}M_{\txtc} \txte^{(\gamma^*-\eta)(i-1-k)} 
	\txte^{-\eta}  K \txte^{-\eta(k-1)}  \| Y^{k-1}-\widetilde{Y}^{k-1},
	(Y^{k-1} -\widetilde{Y}^{k-1})'\|_{\mathcal{D}}.
	\end{align*}
	Again, for the first step of the estimate we make the same deliberations as in the stable case above. 
Furthermore, combining the previous computation and estimating the third summand in~\eqref{j} yields on the center part
	\begin{align*}
	& \sum\limits_{k=0}^{i+1} \txte^{-\eta(i-1)} \Big\|\Big(S^{\txtc}_{\cdot+i-1-k} \Big(T^{\txts}_{R}(\Theta_{k-1}{W}, 
	Y^{k-1}, (Y^{k-1})') [1]- T^{\txts}_{R}(\Theta_{k-1}{W}, 
	\widetilde{Y}^{k-1}, (\widetilde{Y}^{k-1})')[1]\Big),0\Big)\Big\|_{\mathcal{D}}\\
	& + \txte^{-\eta(i-1)} ||\hat{T}^{\txtc}_{R}(\Theta_{i-1}{W}, Y^{i-1}, 
	(Y^{i-1})') [\cdot]- \hat{T}^{\txtc}_{R}(\Theta_{i-1}{W}, 
	\widetilde{Y}^{i-1}, (\widetilde{Y}^{i-1})')[\cdot] ||_{\mathcal{D}}\\
	& \leq \sum\limits_{k=0}^{i} \txte^{(\gamma^*-\eta)(i-1-k)} K C_S(M_{\txtc} 
	\txte^{-\eta} +1) \txte^{-\eta(k-1)} \| Y^{k-1}-\widetilde{Y}^{k-1},(Y^{k-1} 
	-\widetilde{Y}^{k-1})'\|_{\mathcal{D}}\\
	& \leq K C_S	 \frac{\txte^{\gamma^*-\eta} (M_{\txtc}\txte^{-\eta}+1)}{1
		-\txte^{-(\gamma^*-\eta)}} \|\mathbb{Y}-\widetilde{\mathbb{Y}}, \mathbb{Y}'-\widetilde{\mathbb{Y}}'
	\|_{BC^{\eta}(\mathcal{D})}.
	\end{align*}
	Due to~\eqref{gap:k} we have that 
	\begin{align*}
	\|J_{d}({W},\mathbb{Y},\xi) - J_{d}({W},\mathbb{\widetilde{Y}},
	\xi)\|_{BC^{\eta}(\mathcal{D})} \leq \frac{1}{4} \|\mathbb{Y} - 
	\widetilde{\mathbb{Y}}, \mathbb{Y}'-\widetilde{\mathbb{Y}}'\|_{BC^{\eta}(\mathcal{D})}.
	\end{align*}
	Applying Banach's fixed-point theorem, we infer that $J_{d}({W},
	\mathbb{Y},\xi^{\txtc})$ possesses a unique fixed-point $\Gamma(\xi^{\txtc},
	{W})\in BC^{\eta}(\mathcal{D})$ for each fixed $\xi^{\txtc}\in \cB^{\txtc}$. 
	\qed
\end{proof}\\

Theorem~\ref{thm:fp} entails the existence of  $\Gamma(\xi^{\txtc},W)\in BC^{\eta}(\mathcal{D})$ for each fixed $\xi^{\txtc}\in \cB^{\txtc}$.  We denote by $\B_{\cB^{\txtc}}(0,r({W}))$ a ball of $\cB^{\txtc}$, which is centered in $0$ and has a random  radius $r({W})$ and emphasize that the fixed point obtained in Theorem~\ref{thm:fp} characterizes the local center manifold of~\eqref{rpde}. The proof of the next statement is analogue to~\cite[Lemma~4.13]{KN}.

\begin{lemma}
\label{local:cman} 
Under the same assumptions as in Theorem~\ref{thm:fp}, there exists a tempered from below random variable $\widetilde{r}({W})$ such that the local center manifold of~\eqref{rpde} can be represented by
	\begin{equation}
	\cM^{\txtc}_{loc} ({W})=\{ \xi + h^{\txtc}(\xi,{W}) 
	: \xi\in \B_{\cB^{\txtc}}(0,r({W})) \},
	\end{equation}
	where we define
	\begin{align*}
	&	h^{\txtc}(\xi,{W}):=P^{\txts}\Gamma(\xi,{W})[-1,1]|_{\B_{\cB^{\txtc}}(0,r({W}))},
	\end{align*}
	and consequently
	\begin{align*}
	h^{\txtc}(\xi,{W}) &= \sum\limits_{k=-\infty}^{0}S^{\txts}_{-k} 
	\int\limits_{0}^{1} S^{\txts}_{1-r} P^{\txts}F(\Gamma(\xi,{W})[k-1,r]) ~\txtd r\\
	& +\sum\limits_{k=-\infty}^{0}S^{\txts}_{-k} \int\limits_{0}^{1} 
	S^{\txts}_{1-r} P^{\txts}G(\Gamma(\xi,{W})[k-1,r]) ~\txtd \Theta_{k-1}\bm{{W}}_{r}.
	\end{align*}
\end{lemma}

Extending these results to continuous-time dynamical systems as discussed in~\cite{KN} one obtains:
 
\begin{theorem}
\label{manifold} 
Under the assumptions of Theorem~\ref{thm:fp}, there exists a local center manifold for~\eqref{rpde} given by the graph of the function
\begin{align*}
	h^{\txtc}(\xi,{W}) =\int\limits_{-\infty}^{0} S^{\txts}_{-r} 
	P^{\txts}F(U_{r}(\xi)) ~\txtd r + \int\limits_{-\infty}^{0} S^{\txts}_{-r} 
P^{\txts}	G(U_{r}(\xi))~\txtd \bm{{W}}_{r}.
\end{align*}
\end{theorem}

\section{Examples}
\label{appl}

In this section we discuss the applicability of Theorem~\ref{manifold}. This theorem yields the existence of local center manifolds for semilinear rough parabolic PDEs once coefficients satisfy:

\begin{itemize}
\item [1)] the linear part generates an analytic $C_0$-semigroup and its spectrum satisfies~\eqref{ass:linearpart}; 
\item [2)] the drift and the diffusion coefficients $F$ and $G$ satisfy assumptions ({\bf F}) and ({\bf G}) and can be truncated in a neighborhood of the origin such that the gap condition~\eqref{gap:k} holds true, see e.g.~\cite{Caraballo}. 
\end{itemize}
It would be interesting to investigate if the techniques developed in this work can be generalized to rough quasilinear parabolic equations (see~\cite{Simonett} for a deterministic theory) based on the results established in~\cite{HocquetN}.

\begin{example}(Reaction-diffusion type equations with Dirichlet boundary conditions)\label{example1}
	We consider on a monotone scale of interpolation spaces $(\cB_\alpha)_{\alpha\in[0,1]}$, as specified below, the parabolic PDE with zero Dirichlet boundary conditions on the bounded one-dimensional domain $\cO:=[0,\pi]$
	\begin{align}\label{eq:ex:1}
	\begin{cases}
	\txtd u = (\Delta u + u +F(u))~\txtd t + G(u)~\txtd \textbf{W}_{t},\\
	u(0,t)=u(\pi,t)=0, ~\mbox{for } t \geq 0,\\
	u(x,0)=u_{0}(x)\in \cB_\alpha,~\mbox{for } x\in \cO.
	\end{cases}
	\end{align}
In contrast to Example~\ref{ex:cm:strat}, the random input  $\textbf{W}:=(W,\mathbb{W})$ is a $\gamma$-H\"older rough path, for $\gamma\in(1/3,1/2]$. In this case we can construct a Banach scale starting from the operator $Au:=\Delta_D u+u$, where $\Delta_D$ denotes the Dirichlet-Laplacian, as follows. We set $\cB:=L^p(\cO)$, for $1<p<\infty$, $\cB_1:=D(A)=W^{2,p}(\cO)\cap W^{1,p}_0(\cO)$ and $\cB_\alpha=[\cB,\cB_1]_{\alpha}=W^{2\alpha,p}_0(\cO)$, for $\alpha\in[0,1]$. Furthermore, the spectrum of $A$  is constituted by $\{1-n^{2}\mbox{ : } n\geq 1\}$  with corresponding eigenvectors $\{\sin(nx)\mbox{ : } n\geq 1\}$. These give us the
center subspace $\cB^{\txtc}:=\mbox{span}\{\sin x\}$ and the stable one  $\cB^{\txts}:=\mbox{span}\{\sin(nx)\mbox{ : } n\geq 2\}$. \\
The drift term $F:\cB_\alpha\to \cB_{\alpha-\delta}$ is supposed to be locally Lipschitz with linear growth and $G:\cB_{\alpha}\to \cB_{\alpha-\sigma}$ satisfies assumption ({\bf G}). 
Possible choices of $G$ are integral operators obtained as a convolution with a smooth kernel as considered in~\cite[Section 7]{HesseNeamtu1}. Naturally, a linear operator of the form $G(u):=g(x)(-\Delta)^{\sigma}u$ for a smooth function $g$ satisfies assumption ({\bf G}). Here $(-\Delta)^\sigma:\cB_\alpha\to \cB_{\alpha-\sigma}$ for all $\alpha\in\R$ and the multiplication with a smooth function $g$ is a smooth operation from $\cB_{\alpha-\sigma}$ into itself. In this case, we know according to~\cite[Theorem 3.9]{HN21} that~\eqref{eq:ex:1} has a global-in-time solution, therefore the center manifold theory developed in this paper covers this example.
\end{example}

\begin{remark}
Regarding~\ref{ex:cm:strat} it would be desirable to choose a dissipative cubic term for the drift, i.e. $F(u):=-au^3$, for $a>0$.
In order to ensure global-in-time existence of solutions for~\eqref{eq:ex:1}, the drift term $F$ must compensate the stochastic terms. For many classes of stochastic reaction-diffusion equations, this has been proven for additive Brownian noise~\cite{Chow}. Results regarding global-in-time existence for rough differential equations with a dissipative drift term have been obtained in~\cite{Weber}. It should be possible to extend these results to rough PDEs using energy estimates and the equivalence between weak and mild solutions~\cite[Theorem 2.18]{GHN}.	
\end{remark}


 \begin{remark}\label{ho}
 We can easily generalize the previous example to higher-order uniformly elliptic differential operators. Let $m,n\in\N$ and $\cO=[0,\pi]$ and consider
 \begin{align*}
& A u =\sum\limits_{|k|\leq 2m} a_k(x)\txtD ^k u, ~~x\in\cO\\
& \txtD^k u=0, ~~~\mbox{ on } \partial G, ~|k|<m.
 \end{align*}
 The coefficients $a_k\in C^\infty(\overline{\cO})$ and satisfy a uniform ellipticity condition, i.e.~there exists a constant $\overline{c}>0$ such that
 \begin{align*}
 (-1)^m \sum\limits_{|k|=2m} a_k(x)\xi_k\geq \overline{c} |\xi|^{2m}, ~~x\in\cO, \xi\in\R^n.
 \end{align*}
 In this case we choose the spaces $\cB=L^p(\cO)$ for $1<p<\infty$, $\cB_1=D(A)=W^{2mp}(\cO)\cap W^{m,p}_0(\cO)$ and $\cB_\alpha=[\cB,\cB_1]_\alpha=W^{2\alpha m , p }_0(\cO)$.
 	It is known that $A$ has a compact resolvent and therefore countably many eigenvalues $\{\lambda_j\}$ which have finite multiplicities and $\lambda_j\to -\infty$ as $j\to\infty$. 
 	Let $\overline{\lambda}$ be the largest negative eigenvalue of $A$. Therefore the linear operator $L:=A-\overline{\lambda}~\text{Id}$ on $\cB$ with $D(L)=D(A)$ satisfies the assumption~\ref{ass:linearpart}.
 \end{remark}

\begin{example}(Reaction-diffusion type equations on the torus)
We consider the rough PDE with periodic boundary conditions on the one dimensional torus $\mathbb{T}$
\begin{align*}
\begin{cases}
	\txtd u = (\Delta u +F(u))~\txtd t + G(u)~\txtd \textbf{W}_{t}\\
	u(0)=u_0 \in \cB.
\end{cases}
\end{align*}
 Here we work on the scale of Bessel potential spaces  $\cB=H^{k,p}(\mathbb{T})$ for $1<p<\infty$, $k>\frac{1}{p}$ and define $Au:=\Delta u $ with $\cB_1=D(A)=H^{k+2,p}(\mathbb{T})$. Therefore we obtain the scale $\cB_\alpha=H^{k+2\alpha,p}(\mathbb{T})$.
  Note that we consider here only one spatial dimension in order to ensure the gap condition.
 The one-dimensional torus is simply a circle of some given length $\mathbb{T} = R/l\mathbb{Z}$ for $l\in\R$. In this case, the spectrum of $A$ is given by $\Big\{-\Big(\frac{2\pi k}{l}\Big)^2 : k\in\mathbb{Z}\Big\}$. The eigenfunctions corresponding to $0\in\sigma(A)$ are the constant functions which build the center space $\cB^\txtc$.
 Furthermore, we consider $G(u)=p(u)$, where $p$ is a polynomial with smooth coefficients. If the degree of $p$ is greater than one, in order for $G$ to be a smooth operator acting from $\cB_\alpha\to \cB_{\alpha-\sigma}$ for $0\leq\sigma<\gamma$ and for all $\alpha\geq -2\gamma$, we need that $\cB_{-2\gamma}$ is an algebra. Here we recall that $\gamma\in(\frac{1}{3},\frac{1}{2})$ stands for the time-regularity of the rough path. In conclusion we need that $\cB_{-2\gamma}=H^{k-4\gamma,p}(\mathbb{T})$ is an algebra, which is true for $k>\frac{1}{p}+4\gamma$.
 This means that it useful to take as low as possible rough path regularity $\gamma$, as seen in~\cite{GHairer}. Again we assume that we have a dissipative drift which compensates the stochastic terms, in order to guarantee the global-in-time existence of solutions.

\end{example}

Due to its importance to bifurcation theory, see for e.g.~\cite{BloemkerHairer}, we particularly point out the following example which fits into the framework of this work, recall Remark~\ref{ho}.
\begin{example}(Swift-Hohenberg equation with periodic boundary conditions)
We consider
	\begin{align*}
	\begin{cases}
	\txtd u = [ A u + F(u) ]~\txtd t + G(u)~\txtd \textbf{W}_{t}\\
	u(0)=u_0 \in \cB,
	\end{cases}
	\end{align*}
	subject to periodic boundary conditions on the interval $[0,2\pi]$. Here $Au:= -(1+\Delta)^2 u $ and $F(u)=-u^3$.
	 We choose the function spaces $\cB=L^2_{\text{per}}((0,2\pi))$ and $D(A)=H^{4}_{\text{per}}((0,2\pi))$ where per refers to periodic functions. It is well known that $A$ generates an analytic semigroup on $\cB$ and we introduce for simplicity the interpolation spaces $\cB_\alpha=H^{4\alpha}_{\text{per}}((0,2\pi))$. The spectrum of $A$ consists of isolated eigenavlues with finite multiplicities, i.e.~$\sigma(A)=\{ -(1-n^2)^2 : n\in\N\}$.  The eigenvalue $0\in\sigma(A)$ has multiplicity two and $\{e^{ix}\}$ are the corresponding eigenfunctions. Therefore $\cB^\txtc =\text{span}\{\sin x, \cos x \}$ forms the center space.
\end{example}



\begin{thebibliography}{10}
	
	\bibitem{Amann}
	H.~Amann.
	\newblock {\em Linear and quasilinear parabolic problems}.
	\newblock Birkh\"auser Verlag, 1995.
	
	\bibitem{Arnold}
	L.~Arnold.
	\newblock{\em Random Dynamical Systems}. Springer, Berlin Heidelberg, Germany, 2003.
	
	\bibitem{B}
	I.~Bailleul.
	\newblock {\em Flows driven by Banach space-valued rough paths}.
	\newblock {S\'eminaire de Probabilit\'es XLVI, pp. 195--205}, 2014.
	
	
	\bibitem{BailleulRiedelScheutzow}
	I.~Bailleul, S.~Riedel and M.~Scheutzow.
	\newblock Random dynamical system, rough paths and rough flows.
	\newblock {\em J. Differ. Equations}, 262(12):5792--5823, 2017.
	
	\bibitem{BatesJones}
	P.W. Bates and C.K.R.T. Jones.
	\newblock Invariant manifolds for semilinear partial differential equations.
	\newblock In U.~Kirchgraber and H.O. Walther, editors, {\em Dynamics Reported},
	volume~2, pages 1--37. Wiley, 1989.
	
	\bibitem{BloemkerHairer}
	D.~Bl\"omker and M. Hairer.
	\newblock Amplitude equations for SPDEs: Approximate
	centre manifolds and invariant measures. 
	\newblock{\em Probability and partial
		differential equations in modern applied mathematics}. Springer, 2005.
	
	\bibitem{B1}
	D.~Bl\"omker and A.~Neam\c tu. 
	\newblock Amplitude equations for SPDEs driven by fractional additive noise for small Hurst parameter.
	\newblock {\em arXiv:2109.09387}.
	
	
	\bibitem{BloemkerWang}
	D.~Bl\"omker and W.~Wang.
	\newblock{Qualitative properties of local random invariant manifolds for SPDEs
		with quadratic nonlinearity}.
	\newblock{\em J. Dyn. Differ. Equations}, 22(4):677--695, 2010.
	
	\bibitem{B2}
	A.~Blumenthal, M.~Engel and A.~Neam\c tu.
	\newblock On the pitchfork bifurcation for the Chafee-Infante equation with additive noise.
	\newblock {\em arXiv:2108.11073}.
	
	\bibitem{Weber}
	T.~Bonnefoi, A.~Chandra, A.~Moinat and H.~Weber.
	\newblock A priori bounds for rough differential equations with a non-linear damping term.
	\newblock {\em arXiv:2011.06645}.
	
	
	\bibitem{Boxler1}
	P.~Boxler.
	\newblock A stochastic version of the center manifold theory.
	\newblock {\em Probab. Theory Related Fields}. 83(4):509--545, 1989.
	
	\bibitem{Boxler2}
	P.~Boxler.
	\newblock How to construct stochastic center manifolds on the level of
	vector fields. 
	\newblock {\em Lecture Notes in Mathematics}. 1486:141--158, 1991.
	
	\bibitem{CaraballoDuanLuSchmalfuss}
	T.~Caraballo, J.~Duan, K.~Lu, B.~Schmalfu\ss{}.
	\newblock Invariant manifolds for random and stochastic partial differential equations.
	\newblock \emph{Adv. Nonlinear Studies}, 10(1):23--52, 2010. 
	
	\bibitem{Caraballo}
	T.~Caraballo, J.A.~Langa and J.C.~Robinson. 
	\newblock Stability and random attractors for a reactiondiffusion equation with multiplicative noise.
	\newblock {\em Discrete Contin. Dynam. Systems}, 6(4):875--892,
	2000.

	
	\bibitem{Hofmanova2}
	 J.~Cardona, M.Hofmanov\'a, T. Nilssen and N. Rana.
	 \newblock  Random dynamical system generated by
	the 3D Navier--Stokes equation with rough transport noise. 
	\newblock {\em arXiv:2104.14312}, pages 1--27, 2021.
	
	\bibitem{Carr}
	\newblock J.~Carr.
	\newblock{\em Applications of Centre Manifold Theory}. Springer, 1981.
	
	
	
	\bibitem{CastaingValadier}
	C.~Castaing, M.~Valadier. 
	\newblock {\em Convex analysis and measurable multifunctions}.
	\newblock Springer-Verlag, Berlin, 1977. Lecture Notes in Mathematics, Vol. 580.
	
	\bibitem{ChekrounLiuWang}
	M.D.~Chekroun, H.~Liu, S.~Wang.
	\newblock{\em Approximation of stochastic invariant manifolds. 
		Stochastic manifolds for nonlinear SPDEs I}.
	\newblock{Springer, 2015}.
	
	\bibitem{ChenRobertsDuan}
	X.~Chen, A.J.~Roberts, J.~Duan.
	\newblock Centre manifolds for stochastic evolution equations.
	\newblock {\em J. Differ. Equ. Appl.}, 21(7):602--632.  	
	
	\bibitem{Chow}
	P.-L. Chow.
	\newblock{\em Stochastic Partial Differential Equations}.
	\newblock{{Chapman \& Hall / CRC}, 2007}.
	
	\bibitem{ChowLuSell}
	S.-N.~Chow, K.~Lu and G.R.~Sell.
	\newblock Smoothness of inertial manifolds.
	\newblock {\em J.~Math.~Anal.~Appl.}, 169(1):283--312, 1992.
	
	\bibitem{CoutinLejay}
	L.~Coutin and A.~Lejay.
	\newblock{Sensitivity of rough differential equations:
		an approach through the Omega lemma}.
	\newblock{\em arXiv:1712.04705v1}, pages 1--, 2017.
	
		\bibitem{DaPratoZabczyk}
	G.~Da Prato and J. Zabczyk.
	\newblock {\em Stochastic Equations in Infinite Dimensions}.
	\newblock CUP, 2014.
	
	\bibitem{DeyaGubinelliTindel}
	A.~Deya, M.~Gubinelli and S.~Tindel.
	\newblock{Non-linear rough heat equations}.
	\newblock{\em Probab. Theory Related Fields}.
	\newblock 153(1--2):97--147, 2012.
	
	\bibitem{DuDuan} 
	A.~Du and J.~Duan.
	\newblock Invariant manifold reduction for stochastic dynamical systems.
	\newblock {\em Dynamic Systems and Applications} 16:681-696, 2007.
	
	\bibitem{DuanLuSchmalfuss}
	J.~Duan, K.~Lu, B.~Schmalfu\ss{}.
	\newblock Smooth Stable and Unstable Manifolds for Stochastic
	Evolutionary Equations.
	\newblock {\em J. Dynam. Diff. Eq.}, 16(4) 949--972, 2004.
	
	\bibitem{DuanWang}
	J.~Duan and W.~Wang.
	\newblock{\em Effective Dynamics of Stochastic Partial Differential Equations}.
	\newblock Elsevier, 2014.
	
	\bibitem{FehrmanGess}
	B.~Fehrman and B.~Gess.
	\newblock Well-posedness of stochastic porous media equations with nonlinear, conservative noise. To appear in {\em Arch.~Ration.~Mech.~Anal}.
	
	\bibitem{Fenichel}
	N.~Fenichel.
	\newblock Persistence and smoothness of invariant manifolds for flows.
	\newblock{\em J. Indiana Math.}, 21(3):193--226, 1972.
	
	\bibitem{FritzHairer}
	P.K.~Friz and M.~Hairer.
	\newblock {\em A Course on Rough Paths}.
	\newblock Second ed., Springer, 2020.
	
	\bibitem{FritzVictoir}
	P.K.~Friz and N.B.~Victoir.
	\newblock{\em Multidimensional Stochastic Processes as
		Rough Paths: Theory and Applications}.
	Cambridge Studies in Advanced Mathematics, 2010.
	
	\bibitem{FuBloemker}
	H.~Fu and D.~Bl\"omker.
	\newblock {The impact of multiplicative noise in SPDEs close to bifurcation via amplitude equations}.
	\newblock {\em Nonlinearity}, 33(8):3905, 2020.
	
	\bibitem{Gallay}
	T.~Gallay.
	\newblock A center-stable manifold theory for differential equations in Banach spaces.
	\newblock{Comm. Math. Phys.}, 152:249--2689, 1993.
	
	\bibitem{GarridoLuSchmalfuss}
	M.J.~Garrido-Atienza, K. Lu and B. Schmalfu\ss{}.
	\newblock Unstable invariant manifolds for stochastic PDEs driven by a fractional Brownian motion.
	\newblock {\em J. Differ.~Equations}, 248(7):1637--1667, 2010.
	
	\bibitem{GarridoLuSchmalfuss2}
	M.J.~Garrido-Atienza, K.~Lu and B.~Schmalfu\ss{}.
	\newblock{Random dynamical systems for stochastic evolution equations driven by multiplicative fractional Brownian noise with Hurst parametes $H\in(1/3,1/2]$}.
	\newblock{\em SIAM J. Appl. Dyn. Syst}. 15(1), 625--654, 2016.
	
	\bibitem{GHairer}
	A.~Gerasimovics and M.~Hairer.
	\newblock H\"ormander's theorem for semilinear SPDEs.
	\newblock {\em Electron. J. Probab.}, 24:1--56, 2019.
	
	\bibitem{GHN}
	A.~Gerasimovics, A.~Hocquet and T.~Nilssen.
	\newblock Non-autonomous rough semilinear PDEs and the
	multiplicative Sewing Lemma.
	\newblock {\em J.~Func.~Anal.}, 281(10):109200, 2021.
	
	\bibitem{Riedel1}
	M. Ghani Varzaneh, S. Riedel and M. Scheutzow.
	\newblock A dynamical theory for singular stochastic delay differential equations I: Linear equations and a Multiplicative Ergodic Theorem on fields of Banach spaces.
	\newblock {\em 	arXiv:1903.01172}, pages 1--47, 2019.
	\bibitem{Riedel2}
	 M.~Ghani Varzaneh, S.~Riedel and M.~Scheutzow.
	 \newblock A dynamical theory for singular stochastic delay differential equations II: Nonlinear equations and invariant manifolds.
	 \newblock {\em Discrete Contin. Dyn. Syst. B}, 26(8):4587--4612, 2021.
	
	\bibitem{Gubinelli}
	M.~Gubinelli.
	\newblock Controlling rough paths.
	\newblock {\em J. Func. Anal}. 216(1):86--140, 2004.
	
	\bibitem{GubinelliLejayTindel}
	M. Gubinelli, A.~Lejay and S.~Tindel.
	\newblock Young integrals and SPDEs.
	\newblock {\em Potential Anal}. 25(4):307--326, 2006.
	
	\bibitem{GubinelliTindel}
	M.~Gubinelli and S.~Tindel.
	\newblock Rough evolution equations.
	\newblock {\em Ann. Probab}. 38(1):1--75, 2010.
	
	\bibitem{GH}
	J.~Guckenheimer and P.~Holmes.
	\newblock {\em Nonlinear Oscillations, Dynamical Systems, and Bifurcations of Vector Fields}. \newblock Springer, 1983.
	
	\bibitem{Hairer:i}
	M.~Hairer.
	\newblock An introduction to stochastic PDEs.
	\newblock {\em arXiv:0907.4178}, pages 1--78, 2009.
	
	\bibitem{Hairer}
	M.~Hairer.
	\newblock A theory of regularity structures. 
	\newblock {\em Invent. Math.}, 198(2):269--504, 2014.
	
	\bibitem{HairerE}
	M.~Hairer.
	\newblock Ergodicity of stochastic differential equations driven by fractional Brownian motion.
	\newblock {\em Ann. Probab.}, 33(2):703--758, 2005.
	


	
	\bibitem{Henry}
D.~Henry.
\newblock {\em Geometric Theory of Semilinear Parabolic Equations}.
\newblock Springer, Berlin Heidelberg, Germany, 1981.

	\bibitem{HesseNeamtu1}
R.~Hesse and A.~Neam\c tu.
\newblock Local mild solutions for rough stochastic partial differential equations. 
\newblock {\em J. Differential Equat.}, 267(11):6480--6538, 2019.



	\bibitem{HesseNeamtu2}
	R.~Hesse and A.~Neam\c tu.
	\newblock Global solutions and random dynamical systems for rough evolution equations.
\newblock {\em Discrete Contin. Dyn. Syst.}, 25(7):2723--2748, 2020.


\bibitem{HN21}
R.~Hesse and A.~Neam\c tu.
\newblock Global solutions for semilinear rough partial differential equations.
\newblock {\em arXiv:2107.13342}, pages 1--15, 2021.
	
\bibitem{HuNualart}
Y. Hu, D. Nualart, Rough path analysis via fractional calculus. {\em Trans. Am. Math. Soc.}, 361(5):2689--2718, 2009.	

\bibitem{HocquetN}
A.~Hocquet and A.~Neam\c tu. Quasilinear rough evolution equations.
\newblock In preparation.

\bibitem{Hofmanova1}
 M. Hofmanov\'a, J.-M. Leahy and T. Nilssen. On a rough perturbation of the Navier-Stokes
system and its vorticity formulation.
\newblock {\em Annals Appl. Probab.}, 31(2), 2021.	
	
\bibitem{Kuehn}
C.~Kuehn.
\newblock {\em Multiple Time Scale Dynamics}.
\newblock Springer, Berlin Heidelberg, Germany, 2015.

\bibitem{KN} C.~Kuehn and A.~Neam\c tu.  
\newblock Rough center manifolds.
\newblock {\em 
	SIAM J.~Math.~Anal.}, 53(4):3912--3957, 2021.

		
	\bibitem{Kuznetsov}
Y.A. Kuznetsov.
\newblock {\em Elements of Applied Bifurcation Theory}.
\newblock Springer, Berlin Heidelberg, Germany, 2013.

	\bibitem{Kunita}
	H.~Kunita.
	\newblock {\em Stochastic flows and stochastic differential equations}. \newblock Cambridge University Press, 1990.
	
	\bibitem{LianLu}
	Z.~Lian and K.~Lu.
	\newblock Lyapunov exponents and invariant manifolds for random dynamical systems.
	\newblock {\em Mem.  Amer.  Math.  Soc.}, \textbf{206}(2010), vi+106 pp.
	
	\bibitem{LuSchmalfuss}
	K.~Lu and B.~Schmalfu\ss{}.
	\newblock Invariant manifolds for stochastic wave equations.
	\newblock {\em J. Differential Equat.}, 236(2):460--492, 2007.
	
	\bibitem{Lyons}
	T.~Lyons.
	\newblock Differential equations driven by rough signals.
	\newblock{\em Rev. Mat. Iberoamericana}, 14(2):215--310, 1998.
	
	\bibitem{Lunardi}
	A.~Lunardi.
	\newblock{\em Analytic semigroups and optimal regularity in parabolic problems}.
	\newblock Birkh\"auser, 1995.
	
	\bibitem{MaslowskiNualart}
	B.~Maslowski and D.~Nualart.
	\newblock Evolution equations driven by a fractional Brownian motion.
	\newblock {\em J. Funct. Anal.}, 202(1):277--305, 2003.
	
	\bibitem{Mohammed}
	S.~Mohammed, T.~Zhang and H.~Zhao.
	\newblock {\em The stable manifold theorem for semilinear stochastic evolution equations and stochastic partial differential equations}.
	\newblock Memoirs of the AIMS, vol.~196, nr.~197, 2008.
	
	\bibitem{Neamtu}
	A.~Neam\c tu. Random invariant manifolds for ill-posed stochastic evolution equations.
	\newblock {\em Stoch.~Dynam.}, 20(2):2050013, 2020.
	
	\bibitem{Pazy}
	A.~Pazy.
	\newblock{\em Semigroups of Linear Operators and Applications to Partial Differential Equations}.
	\newblock{Springer Applied Mathematical Series. Springer--Verlag, Berlin, 1983}.
	
	\bibitem{Roberts} 
	A.J.~Roberts.
	\newblock Normal form transforms separate slow and fast modes in
	stochastic dynamical systems.
	\newblock{\em Physica A}, 387(1):12-38, 2008.
	
	\bibitem{Scheutzow}
	M.~Scheutzow. 
	\newblock On the perfection of crude cocycles.
	\newblock {\em Random Comput. Dynam.}, 4(4):235--255, 1996.
	
	
	\bibitem{SchoenerHaken}
	G.~Sch\"oner and H.~Haken. 
	\newblock A systematic elimination procedure for Ito stochastic differential
	equations and the adiabatic approximation.
	\newblock {\em Z. Phys.}, 68(1):89--103, 1987.
	
	
	
	\bibitem{SellYou}
	G.~Sell and Y.~You.
	\newblock {\em Dynamics of evolutionary equations}.
	\newblock Springer-Verlag New York, 2002.
	
	\bibitem{Simonett}
	G.~Simonett. 
	\newblock{Center manifolds for quasilinear reaction-diffusion systems}.
	\newblock {\em Differ. Integral Equ.}, 8(4):753--796, 1995.

	\bibitem{v}
	\newblock A. Vanderbauwhede.
	\newblock Center manifold, normal forms and elementary bifurcations.
	\newblock{\em Dynamics
		Reported}. pp. 89--169, 1989.
	
	\bibitem{vioos}
	A.~Vanderbauwhede and G.~Iooss.
	\newblock Center manifold theory in infinite dimensions. 
	\newblock{\em Dynamics Reported}. Springer-Verlag Berlin, pp. 125--163, 1992.
	
	\bibitem{WanngDuan}
	W.~Wang and J.~Duan.
	\newblock A dynamical approximation of stochastic partial differential equations.
	\newblock{\em J. Math. Phys.}, 48(10):102701--102701-14, 2007.
	
	\bibitem{WaymireDuan}
	E.~Waymire and J.~Duan.
	\newblock{\em Probability and partial differential equations in modern applied mathematics}.
	\newblock Springer--Verlag, 2005.
	
	\bibitem{Yagi}
	A.~Yagi.
	\newblock {\em Abstract Parabolic Evolution Equations and their Applications}.
	\newblock Springer, 2010.
	
	
	\bibitem{Young}
	L.~C.~Young.
	\newblock{An integration of H\"older type, connected with Stieltjes integration}.
	\newblock{\em Acta Math}. 67(1):251--282, 1936.
	
\end{thebibliography}
\end{document}